\documentclass{amsart}

\usepackage{amsmath}
\usepackage{amssymb}
\usepackage{amsthm}
\usepackage{amscd}
\usepackage{algorithmic}
\usepackage{algorithm}
\usepackage{array}
\usepackage{bm}
\usepackage{centernot}
\usepackage[dvips]{graphicx}
\usepackage{fancyhdr}
\usepackage{hhline}
\usepackage{longtable}
\usepackage{multicol}
\usepackage{multirow}
\usepackage{supertabular}
\usepackage{slashbox}
\usepackage{tabularx}
\usepackage{url}
\usepackage{verbatim}
\usepackage{comment}
\newtheorem{lem}{Lemma}

\newtheorem{rem}{Remark}

\newtheorem{thm}{Theorem}

\begin{document}

\title{
Some properties of ${\tau}$-adic expansions \\
on hyperelliptic Koblitz curves
}

\author{Keisuke Hakuta}
\address{Yokohama Research Laboratory, Hitachi, Ltd., \\
292, Yoshida-cho, Totsuka-ku, Yokohama, 244-0817, Japan}
\email{keisuke.hakuta.cw@hitachi.com}
\author{Hisayoshi Sato}
\address{Yokohama Research Laboratory, Hitachi, Ltd., \\
292, Yoshida-cho, Totsuka-ku, Yokohama, 244-0817, Japan}
\email{hisayoshi.sato.th@hitachi.com}
\author{Tsuyoshi Takagi}
\address{Institute of Mathematics for Industry, Kyushu University, 
\\744, Motooka, Nishi-ku, Fukuoka, Fukuoka, 819-0395, Japan}
\email{takagi@imi.kyushu-u.ac.jp}


\begin{abstract}
This paper explores two techniques 
on a family of hyperelliptic curves that have been proposed 
to accelerate computation of scalar multiplication 
for hyperelliptic curve cryptosystems. 
In elliptic curve cryptosystems, 
it is known that Koblitz curves admit fast scalar multiplication, 
namely, the ${\tau}$-adic non-adjacent form ($\tau$-NAF). 
It is shown that the $\tau$-NAF has the three properties: 
(1) existence, (2) uniqueness, 
and (3) minimality of the Hamming weight. 
These properties are not only of intrinsic mathematical interest, 
but also desirable in some cryptographic applications. 
On the other hand, 
G{\"u}nther, Lange, and Stein have proposed 
two generalizations of $\tau$-NAF 
for a family of hyperelliptic curves, 
called \emph{hyperelliptic Koblitz curves}. 
However, to our knowledge, 
it is not known whether the three properties are true or not. 
We provide an answer to the question. 
Our investigation shows that 
the first one has only the existence 
and the second one has the existence and uniqueness. 
Furthermore, we shall prove that there exist 16 digit sets 
so that one can achieve the second one. 
\end{abstract}

\maketitle

\section{Introduction}
\label{sec:hyperKC_intro}

In 1989, Koblitz in \cite{Kob89} proposed hyperelliptic curve cryptosystems (HECC)
which is an alternative approach to construct public key cryptosystems. 
As is the case in elliptic curve cryptosystems (ECC), 
scalar multiplication is the dominant operation. 
In order to accelerate the scalar multiplication, 
many efficient techniques have been proposed. 
In the case of ECC, one such technique is to use curves 
over finite fields of characteristic two 
because these curves have high affinity for software/hardware implementations. 
Koblitz in \cite{Kob91} proposed a family of elliptic curves 
for cryptographic use: 
\begin{equation*}
{E_a} : y^2 + xy = x^3 + ax^2 + 1, \quad %
a \in \mathbb{F}_2 
\end{equation*}
over a finite field ${\mathbb{F}_2}$. 
These curves are called Koblitz curves. 
For some cryptographic usage, we focus on 
the group of $\mathbb{F}_{2^m}$-rational points 
${E_a}(\mathbb{F}_{2^m})$ 
for some $m \geq 2$. 
Let $\tau$ be the Frobenius map on $E_a$, 
\begin{equation*}
\tau: {E_a}(\mathbb{F}_{2^m}) \rightarrow {E_a}(\mathbb{F}_{2^m}), %
\quad (x, y) \mapsto (x^{2}, y^{2}). 
\end{equation*}
Especially, Koblitz curves admit fast scalar multiplication, 
namely the $\tau$-adic non-adjacent form ($\tau$-NAF for short) 
proposed by Solinas \cite{Sol00}. 
$\tau$-NAF of $\alpha \in \mathbb{Z}[\tau]$ 
with respect to the digit set $\{0, \pm{1} \}$, 
is a $\tau$-adic expansion 
$\alpha = \sum_{i=0}^{\ell-1} {c_i} \tau^{i}$ such that 
${c_{i+1}}{c_{i}} = 0$ for all $i$ ($0 \leq i \leq \ell - 2$) 
and $c_i \in \{0, \pm{1} \}$ for all $i$ ($0 \leq i \leq \ell - 1$). 
$\tau$-NAF on Koblitz curves has three properties: 

\begin{description}
\item[(1)] Existence (\cite[Theorem~1, Algorithm~1]{Sol00}): 
Every $\alpha \in \mathbb{Z}[\tau]$ 
has a $\tau$-NAF with respect to 
the digit set $\{0, \pm{1} \}$. 
\item[(2)] Uniqueness (\cite[Theorem~1]{Sol00}): 
Each $\alpha \in \mathbb{Z}[\tau]$ 
can be uniquely represented as a $\tau$-NAF 
with respect to the digit set $\{0, \pm{1} \}$. 
\item[(3)] Minimality (\cite[Theorem~1]{AHP05}, \cite[Theorem~1]{AHP06b}, \cite[p.80]{HST_JMI}): 
The $\tau$-NAF has minimal Hamming weight 
among all Frobenius expansions 
with respect to the digit set $\{0, \pm{1} \}$. 
\end{description}

The existence must be satisfied for concrete cryptographic implementations. 
The uniqueness and the minimality are not only of intrinsic mathematical interest, 
but also desirable in some cryptographic applications 
such as batch verification (cf. \cite{CL06}, \cite{CY07}, \cite{HKST12}). 
For higher width version of the $\tau$-NAF on Koblitz curves 
or another special types of elliptic curves, 
these properties have been investigated 
(cf. \cite{AHP06a}, \cite{AHP08}, \cite{AHP10}, %
\cite{BMX08}, \cite{HST10}, \cite{Heu10}, %
\cite{HK10}, \cite{HK11}, \cite{HK12}). 

On the other hand, 
G{\"u}nther, Lange, and Stein in \cite{GLS00}, \cite{GLS01} 
have proposed two generalizations of $\tau$-NAF 
for a family of hyperelliptic curves 
\begin{equation}
{C_a}: y^{2} + xy = x^{5} + ax^{2} + 1, %
\quad a \in \mathbb{F}_{2} \label{hyperKob} %
\end{equation}
over a finite field $\mathbb{F}_{2}$. 
The curves are listed in \cite[Table~1]{Kob89}. 
We call the curves as \emph{hyperelliptic Koblitz curves}. 
We identify $\{0, 1\} (\subset \mathbb{Z})$ 
with ${\mathbb{F}_2}$ via the natural map 
$f:\{0, 1\} \rightarrow {\mathbb{F}_2}$, 
$a \mapsto a \bmod{2}$. 
Let $\mathrm{Jac}({C_a})(\mathbb{F}_{2^m})$ %
($\cong \mathrm{Pic}^{0}(C/\mathbb{F}_{2^m})$) 
be the Jacobian variety of ${C_a}$ 
and ${\tau}$ the Frobenius map on $\mathrm{Jac}({C_a})$. 
We can regard ${\tau}$ as a complex number 
which satisfies the following characteristic equation 
\begin{equation}
{\tau}^{4} - {\mu}{\tau}^{3} - 2{\mu}{\tau} + 4 = 0, %
\quad %
{\rm where} \quad \mu = (-1)^{1-a}. \label{hyperKC_cha_eqn} %
\end{equation}

The two generalizations are as follows: 

\begin{description}
\item[(i)] GLS $\tau$-adic expansion: 
One is the $\tau$-adic expansion with the strategy 
``at least one of four consecutive coefficients is zero''. 
The expansion of $\alpha \in \mathbb{Z}[\tau]$ is 
a $\tau$-adic expansion $\alpha = \sum_{i = 0}^{\ell-1} {c_i}{\tau}^{i}$ 
such that ${c_{i+3}}{c_{i+2}}{c_{i+1}}{c_{i}} = 0$ for all $i$ 
($0 \leq i \leq \ell - 4$) and 
${c_i} \in \mathcal{D} = %
\{0, \pm{1}, \pm{2}, \pm{3} \}$ for all $i$ 
($0 \leq i \leq \ell - 1$). 
We will call the above expansion as \emph{GLS $\tau$-adic expansion}. 
\item[(ii)] $\tau$-NAF: 
The other is a sparse $\tau$-adic expansion. 
The expansion of $\alpha \in \mathbb{Z}[\tau]$ is 
a $\tau$-adic expansion $\alpha = \sum_{i = 0}^{\ell-1} {c_i}{\tau}^{i}$ 
such that ${c_{i+1}}{c_{i}} = 0$ for all $i$ 
($0 \leq i \leq \ell - 2$) and 
${c_i} \in \widetilde{\mathcal{D}} = %
\{ 0, \pm{1}, \pm{2}, \pm{(1 + {\tau})}, %
\pm{(1 - {\tau})}, \pm{(1 - 2{\tau})}, \pm{2} + {\tau} \}$ 
for all $i$ ($0 \leq i \leq \ell - 1$). 
We will call the above expansion as \emph{$\tau$-NAF}. 
\end{description}

For the GLS $\tau$-adic expansion, 
the authors say (\cite[p.9]{GLS00}) 
``Our experiments show that the expansion is always finite. 
However, we were unable to close this final gap so far''. 
In addition, the authors do not give any information 
about the existence of $\tau$-NAF. 
Moreover, to our knowledge, 
there are no proofs of these existences 
in the relevant literature 
(cf. \cite{ACDFLNV05}, \cite{GLS00}, \cite{GLS01}, \cite{Lan01}, \cite{Lan05}). 
This raises the question of whether or not 
the three properties are true for hyperelliptic Koblitz curves. 

\subsection{Contribution of this paper} 
\label{subsec:hyperKC_contribution}

The purpose of this paper is to explore the above questions. 
We show a summary of this paper in Table~\ref{table:hyperKC_summary}. 

\begin{table}[htbp]
\caption{Properties of two $\tau$-adic expansions on hyperelliptic Koblitz curves} 
\begin{center}
\begin{tabular}{|c|c|c|c|}\hline 
Types of expansion & Existence & Uniqueness & Minimality \\\hline\hline
GLS $\tau$-adic expansion & Yes & No & No \\\hline
$\tau$-NAF & Yes & Yes & No \\\hline
\end{tabular}
\end{center}
\label{table:hyperKC_summary}
\end{table}

Our investigation shows that 
the GLS $\tau$-adic expansion has only the existence 
and the $\tau$-NAF has the existence and uniqueness. 
Furthermore, we can prove that there exist 16 digit sets 
so that one can achieve the $\tau$-NAF. 

The rest of this paper is organized as follows. 
In Section~\ref{sec:hyperKC_prop_GLS} and Section~\ref{sec:hyperKC_prop_tnaf}, 
we explore the above questions for the GLS $\tau$-adic expansion 
and for the $\tau$-NAF, respectively. 
Section~\ref{sec:hyperKC_conclusion} concludes the paper. 


\section{Properties of the GLS $\tau$-adic expansion}
\label{sec:hyperKC_prop_GLS}

In this section, we provide the three properties 
of GLS $\tau$-adic expansion on hyperelliptic Koblitz curves. 

\subsection{Algorithm to compute GLS $\tau$-adic expansion}
\label{subsec:hyperKC_gls_review}

We review the algorithm to compute GLS $\tau$-adic expansion. 
First, we prove a lemma that will be useful 
to show the existences of the GLS $\tau$-adic expansion and $\tau$-NAF. 
Whereas the former part of the lemma is already discussed 
in \cite[pp.6--7]{GLS00}, \cite[p.111]{GLS01} 
for only the case $\mu = 1$, the following proof of the lemma covers 
both $\mu = 1$ and $\mu = -1$. 

\begin{lem} {\rm\bfseries [Divisibility by $\tau$]}\ 
\label{lem:hyperKC_div}
Let 
$\alpha = s + t{\tau} + u{\tau}^{2} + v{\tau}^{3} \in {\mathbb{Z}[\tau]}$ 
($s, t, u, v \in \mathbb{Z}$), 
then we have
${\tau} \mid \alpha \iff 4 \mid s$. 
In particular, 
we have ${\tau}^{2} \mid {\alpha} \ \implies \ {4 \mid s}$ 
and $4 \mid ({\mu}s/2 + t)$. 
\end{lem}

\begin{proof}
We first show the former part. 

\noindent
($\implies$) 
We assume ${\tau} \mid \alpha$. 
There exists $s^{\prime}, t^{\prime}, u^{\prime}, %
v^{\prime} \in \mathbb{Z}$ such that 
$\alpha = {\tau}(s^{\prime} + t^{\prime}{\tau} + u^{\prime}{\tau}^{2} + v^{\prime}{\tau}^{3})$. 
Then 
\begin{eqnarray*}
\alpha 
& = & 
{\tau}(s^{\prime} + t^{\prime}{\tau} + u^{\prime}{\tau}^{2} %
+ v^{\prime}{\tau}^{3}) \nonumber \\%
& = &
s^{\prime}{\tau} + t^{\prime}{\tau}^{2} + u^{\prime}{\tau}^{3} %
+ v^{\prime}{\tau}^{4} \nonumber \\%
& = &
s^{\prime}{\tau} + t^{\prime}{\tau}^{2} + u^{\prime}{\tau}^{3} %
+ v^{\prime}(\mu{\tau}^{3} + 2{\mu}{\tau} - 4) \nonumber \\%
& = &
-4v^{\prime} + (s^{\prime} + 2{\mu}v^{\prime}){\tau} + %
t^{\prime}{\tau}^{2} + (u^{\prime} + {\mu}v^{\prime}){\tau}^{3}. %
\end{eqnarray*}
Hence we obtain $s = 4v^{\prime}$. 

\medskip
\noindent
($\impliedby$) 
Conversely, we assume $4 \mid s$. 
There exists $s^{\prime} \in \mathbb{Z}$ 
such that $s = 4s^{\prime}$. 
Then 
\begin{eqnarray*}
\alpha 
& = & 
4s^{\prime} + t{\tau} + u{\tau}^{2} + v{\tau}^{3} \nonumber \\%
& = &
s^{\prime}(-{\tau}^{4} + {\mu}{\tau}^{3} + 2{\mu}{\tau}) %
+ t{\tau} + u{\tau}^{2} + v{\tau}^{3} \nonumber \\%
& = &
{\tau}((2{\mu}s^{\prime} + t) + u{\tau} %
+ ({\mu}s^{\prime} + v){\tau}^{2} - s^{\prime}{\tau}^{3}). %
\end{eqnarray*}
Therefore ${\tau} \mid {\alpha}$. 

Next, assume that ${\tau}^{2} \mid {\alpha}$. 
By the same argument as the first part, 
${\alpha}/{\tau} = (2{\mu}s^{\prime} + t) + u{\tau} %
+ ({\mu}s^{\prime} + v){\tau}^{2} - s^{\prime}{\tau}^{3}$ 
is divisible by $\tau$. 
Thus, $2{\mu}s^{\prime} + t = {\mu}s/{2} + t$ is divisible by $4$. 
\end{proof}

A process to construct GLS ${\tau}$-adic expansion 
is as follows. 
Set ${\alpha}_{i} := {\alpha} = s + t{\tau} + u{\tau}^{2} %
+ v{\tau}^{3} \in {\mathbb{Z}[\tau]}$ 
($s, t, u, v \in \mathbb{Z}$). The initial index $i$ is zero. 
We choose a $c$ according to the following rules: 

\medskip
\noindent
(1) If $4 \mid s$ \ then set $c = 0$. 

\medskip
\noindent
(2-1) If $4 \centernot\mid s$ and  $2 \mid t$ \ 
then set $c$ according to the following table: 

\vspace*{-0.3cm}
\begin{table}[htbp]
\begin{center}
\begin{tabular}{l|cccccc} 
$t \bmod 4 \backslash s \bmod 8$ & 
$1$ & $2$ & $3$ & $5$ & $6$ & $7$ \\\hline
$0$ & $1$ & $2$ & $3$ & $-3$ & $-2$ & $-1$ \\
$2$ & $-3$ & $-2$ & $-1$ & $1$ & $2$ & $3$ \\
\end{tabular}
\end{center}
\end{table}

\medskip
\noindent
(2-2) If $4 \centernot\mid s$ and  $2 \centernot\mid t$ \ 
then set $c$ according to the following table: 

\vspace*{-0.3cm}
\begin{table}[htbp]
\begin{center}
\begin{tabular}{l|cccccc}
$d \bmod 2 \backslash s \bmod 8$ & 
$1$ & $2$ & $3$ & $5$ & $6$ & $7$ \\\hline
$0$ & $1$ & $2$ & $3$ & $-3$ & $-2$ & $-1$ \\
$1$ & $-3$ & $-2$ & $-1$ & $1$ & $2$ & $3$ \\
\end{tabular}
\end{center}
\end{table}

Then put ${\alpha}_{i+1} := ({\alpha}_{i} - c)/{\tau}$, $i:=i + 1$. 
Repeating the process until ${\alpha}_{i}$ will be zero for some $i$, 
leads the following Algorithm. 

\begin{algorithm}[H]
  \caption{GLS ${\tau}$-adic expansion 
}
  \label{alg:hyperKC_GLS}
  \begin{algorithmic}[1]
    \REQUIRE ${\alpha} = s + t{\tau} + u{\tau}^{2} + v{\tau}^{3} \in \mathbb{Z}[\tau]$ ($s, t, u, v \in \mathbb{Z}$)
    \ENSURE ${\tau}$-NAF of $\alpha$
    \STATE $i \leftarrow 0$
    \WHILE{$s \neq 0$ or $t \neq 0$ or $u \neq 0$ or $v \neq 0$}
    \STATE $c \leftarrow s \bmod 4$
    \IF{$c \neq 0$}
    \IF{( ($t \bmod 4 = 0$ and $s \bmod 8 > 4$) \\
or ($t \bmod 4 = 2$ and $s \bmod 8 < 4$) \\
or ($t \bmod 2 = 1$ and $s \bmod 8 > 4$ and $v \bmod 2 = 0$) \\
or ($t \bmod 2 = 1$ and $s \bmod 8 < 4$ and $v \bmod 2 = 1$) )}
    \STATE $c \leftarrow c - 4$
    \ENDIF
    \ENDIF
    \STATE ${c_i} \leftarrow c$
    \STATE $d \leftarrow \frac{\mu(s - c)}{4}$, $s \leftarrow 2d + t$, $t \leftarrow u$, $u \leftarrow d + v$, $v \leftarrow -{\mu}d$
    \STATE $i \leftarrow i + 1$
    \ENDWHILE
    \STATE $\ell \leftarrow i$
    \RETURN $(c_{\ell-1}, \cdots, c_1, c_0)_{\tau}$
  \end{algorithmic}
\end{algorithm}

\subsection{A gap in the proof of the existence}
\label{subsec:hyperKC_gls_gap} 

In \cite[p.9]{GLS01}, the authors tried to prove 
the existence of GLS $\tau$-adic expansion 
by using the Triangle inequality for the usual absolute value. 
If there exists a non negative integer $i$ 
such that $|{\alpha}_{i}| > 3(\sqrt{2} + 1)$, 
then it satisies that $|{\alpha}_{i+1}| < |{\alpha}_{i}|$, 
where $|\cdot|$ is the usual absolute value. 
However, if $|{\alpha}_{i}| \leq 3(\sqrt{2} + 1)$, 
we can not show that there exists a positive integer ${j_0}$ 
such that $|{\alpha}_{i+{j_0}}| < |{\alpha}_{i}|$ 
by using the Triangle inequality. 
The gap in the proof of the existence of GLS $\tau$-adic expansion 
is to prove that (1) there are only finitely many elements 
$\alpha \in {\mathbb{Z}[\tau]}$ such that 
$|{\alpha}| \leq 3(\sqrt{2} + 1)$, 
(2) each element of (1) has GLS $\tau$-adic expansion. 

On the other hand, Lange in \cite{Lan01}, \cite{Lan05} 
generalized Frobenius expansions on subfield elliptic curves 
to Frobenius expansions on subfield hyperelliptic curves. 
Lange explained that the Fincke-Pohst algorithm 
is useful to prove the existence of Frobenius expansions 
on subfield hyperelliptic curves. 
The Fincke-Pohst algorithm (\cite{FP85}) enumerates 
all shorter vectors for a given lattice and an upper bound. 
In order to prove the existence by using the Fincke-Pohst algorithm, 
we use the following consideration. 
Let $C$ be a hyperelliptic curve of genus $g$ over a finite field $\mathbb{F}_{q}$, 
and ${\tau}_{1}, \ldots, {\tau}_{g}$ the $g$ independent roots 
of the characteristic equation of $C/\mathbb{F}_{q}$. Take the set of elements 
\[
\Lambda := \Biggl{\{} \Biggl{(} \sum_{j = 0}^{2g - 1} {c_j}{{\tau}_{1}}^{j}, %
\ldots, \sum_{j = 0}^{2g - 1} {c_j}{{\tau}_{g}}^{j} \Biggr{)} %
\Biggm{|} 
{c_j} \in \mathbb{Z} \Biggr{\}}. %
\]
It is easy to see that $\Lambda$ is a lattice in $\mathbb{C}^g$, 
where $\mathbb{C}$ denotes the complex number field. 
Lange proposed to investigate the norm of vectors in this lattice, 
where the norm is given by the usual Euclidean norm of $\mathbb{C}^g$: 
\[
\mathcal{N}: ({x_1}, \ldots, {x_g}) %
\mapsto \sqrt{|{x_1}|^{2} + \cdots + |{x_g}|^{2}}. %
\]
Although the Fincke-Pohst algorithm exponential time complexity, 
it works effectively in our situation. 
Since $\mathcal{N}(\alpha)^{2}$ is a quadratic form 
in the $2g$ variables ${c_0}, {c_1}, \ldots, {c_{2g-1}}$, 
it is easily verify that for $\alpha \in \mathbb{Z}[\tau]$, 
$\mathcal{N}(\alpha)^{2}$ takes integer value, 
namely, $\mathcal{N}(\alpha)^{2} \in \mathbb{Z}$ 
(see \cite[Example~8.5]{Lan01}). 
Moreover,  $\mathcal{N}$ satisfies the Triangle inequality, 
that is, 
$\mathcal{N}({\alpha_{1}} + {\alpha_{2}}) \leq %
\mathcal{N}({\alpha_{1}}) + \mathcal{N}({\alpha_{2}})$ 
for ${\alpha_{1}}, {\alpha_{2}} \in \mathbb{Z}[\tau]$. 
The above facts can be used to prove the existence. 
We outline of the proof of the existence. 
\begin{description}
\item[(I)] 
We choose an element ${\alpha}_{0}:={\alpha} \in \mathbb{Z}[\tau]$, 
and fixed it. 
We first prove that there exist a positive integer ${j_0}$ 
and some constant $C \in \mathbb{C}$ such that 
$\mathcal{N}({\alpha}_{i+{j_0}}) < \mathcal{N}({\alpha}_{i})$ 
for ${\alpha}_{i} \in \mathbb{Z}[\tau]$ with $\mathcal{N}({\alpha_i}) > C$. 
Since $\mathcal{N}(\alpha)^{2} \in \mathbb{Z}$, 
we can see that the number of elements of 
$\{ {\alpha_i} \mid \mathcal{N}({\alpha_i}) > C \}$ is finite. 
\item[(II)] 
We find all elements in the set 
$\{ {\alpha} \in \mathbb{Z}[\tau] \mid \mathcal{N}({\alpha_i}) \leq C \}$ 
or equivalently 
the set 
$\{ {\alpha} \in \mathbb{Z}[\tau] \mid \mathcal{N}^{2}({\alpha_i}) %
\leq \lfloor C^2 \rfloor \}$, 
where $\lfloor \cdot \rfloor$ is the floor function symbol. 
To do this, we can use the Fincke-Pohst algorithm 
because $\mathcal{N}(\alpha)^{2}$ is a quadratic form. 
\item[(III)] 
For each element which is found in (II), 
we compute its GLS ${\tau}$-adic expansion 
and verify that 
the GLS ${\tau}$-adic expansions have finite length. 
\end{description}

Next, we will fill the gap according to the outline. 

\subsection{Properties}
\label{subsec:hyperKC_gls_properties} 

\begin{thm} 
\label{thm:hyperKC_GLS_existence}
{\rm\bfseries 
[Existence of the GLS ${\tau}$-adic expansion]}\ 
Every $\alpha \in \mathbb{Z}[\tau]$ 
has a GLS ${\tau}$-adic expansion with digit set $\mathcal{D}$. 
\end{thm}

\begin{proof}
Let ${\alpha}_{0} := {\alpha} = s + t{\tau} + u{\tau}^{2} %
+ v{\tau}^{3} \in {\mathbb{Z}[\tau]}$ 
($s, t, u, v \in \mathbb{Z}$). 
First, we claim that there exist 
${\alpha}^{\prime} \in {\mathbb{Z}[\tau]}$, 
${\ell}^{\prime} \in \mathbb{Z}_{>0}$ and 
$c_{0}, c_{1}, \ldots, c_{{\ell}^{\prime} - 1} \in \mathcal{D}$ 
such that $\mathcal{N}({\alpha}^{\prime}) \leq 2 + 3\sqrt{2}$ and 
\[
\alpha = \sum_{j=0}^{{\ell}^{\prime}-1} {c_{j}} + {\alpha}^{\prime}{\tau}^{{\ell}^{\prime}}. %
\]
Notice that 
$\max \{ \mathcal{N}(\alpha) \mid \alpha \in \mathcal{D} \} %
= 3\sqrt{2}$ (the equality holds when 
$\alpha = \pm{3}$). 
We put ${\alpha}_{i + 1} := ({\alpha}_{i} - {c_i})/{\tau}$ 
(${c_i} \in \mathcal{D}$). 
If ${c_i} = 0$, it always satisfies that 
$\mathcal{N}({\alpha}_{i + 1}) \leq \mathcal{N}({\alpha}_{i})/\sqrt{2} %
< \mathcal{N}({\alpha}_{i})$ for ${\alpha}_{i} \neq 0$. 
Otherwise, namely, ${c_i} \neq 0$, 
we put ${\alpha_{i}} = {s_i} + {t_i}{\tau} + {u_i}{\tau}^{2} + {v_i}{\tau}^{3}$ 
(${s_i}, {t_i}, {u_i}, {v_i} \in \mathbb{Z}$). 
We claim that we have $\mathcal{N}({\alpha}_{i + {j_0}}) < \mathcal{N}({\alpha}_{i})$ 
for some ${j_0} \in \{ 1, 2, 3, 4 \}$ and $\mathcal{N}({\alpha}_{i}) > 2 + 3\sqrt{2}$. 
To see this, there are three cases to consider (1) ${t_i} = 0 \bmod 2$, 
(2) ${t_i} = 1 \bmod 2$ and ${u_i} = 0 \bmod 2$, 
and (3) ${t_i} = 1 \bmod 2$ and ${u_i} = 0 \bmod 2$. 

\medskip
\noindent
{\bf Case~1. } ${t_i} = 0 \bmod 2$. \\
In this case we have ${\alpha}_{i + 2} = ({\alpha}_{i} - {c_i})/{\tau}^{2}$. 
Taking the norm of the both sides in the equation 
and applying the Triangle inequality gives 
$\mathcal{N}({\alpha}_{i + 2}) < \mathcal{N}({\alpha}_{i})$ 
for $\mathcal{N}({\alpha}_{i}) > 3\sqrt{2}$. 

\medskip
\noindent
{\bf Case~2. } ${t_i} = 1 \bmod 2$ and ${u_i} = 0 \bmod 2$. \\
In this case we have 
${\alpha}_{i + 3} = %
(({\alpha}_{i} - {c_i})/{\tau} - {c_{i + 1}})/{\tau}^{2}$. 
As in case~1, 
we obtain 
$\mathcal{N}({\alpha}_{i + 3}) < \mathcal{N}({\alpha}_{i})$ 
for $\mathcal{N}({\alpha}_{i}) > (18 + 15\sqrt{2})/7$. 

\medskip
\noindent
{\bf Case~3. } ${t_i} = 1 \bmod 2$ and ${u_i} = 1 \bmod 2$. \\
One has immediately that ${\alpha}_{i + 4} = %
( ( ({\alpha}_{i} - {c_i})/{\tau} - {c_{i + 1}})/{\tau} - {c_{i + 2}})/{\tau}^{2}$. 
The same argument as case~1 and case~2 yields 
$\mathcal{N}({\alpha}_{i + 4}) < \mathcal{N}({\alpha}_{i})$ 
for $\mathcal{N}({\alpha}_{i}) > 2 + 3\sqrt{2}$. 

\medskip
\noindent
Note that $3\sqrt{2} < (18 + 15\sqrt{2})/7 < 2 + 3\sqrt{2} \approx 6.24$. 
Hence the claim follows. 

Next, we find all elements 
${\alpha}^{\prime} \in {\mathbb{Z}[\tau]}$ 
satisfying $\mathcal{N}({\alpha}^{\prime}) \leq 2 + 3\sqrt{2}$, 
namely $\mathcal{N}({\alpha}^{\prime})^{2} \leq 38$ %
($< (2 + 3\sqrt{2})^{2} \approx 38.97$). 
We use the Fincke-Pohst algorithm to find them. 

\begin{center}
\begin{longtable}[c]{|c||c|c|c|c|}
\caption{$\alpha^{\prime} \in \mathbb{Z}[{\tau}]$ %
with $\mathcal{N}(\alpha^{\prime})^{2} \leq 38$ ($\mu = 1$). }
\\
\hline
\label{table:hyperKC_tnaf_existence_a=1}
${\#}$ & ${\alpha}^{\prime} = s + t{\tau} + u{\tau}^{2} + v{\tau}^{3}$ & %
$\mathcal{N}(\alpha^{\prime})^{2}$ & GLS ${\tau}$-adic expansion of $\alpha^{\prime}$ & $\ell(\alpha^{\prime})$ \\
\hline\hline
\endfirsthead
\hline
${\#}$ & ${\alpha^{\prime}} = s + t{\tau} + u{\tau}^{2} + v{\tau}^{3}$ & %
$\mathcal{N}(\alpha^{\prime})^{2}$ & GLS ${\tau}$-adic expansion of $\alpha^{\prime}$ & $\ell(\alpha^{\prime})$ \\
\hline\hline
\endhead
\hline
\endfoot
\hline
\endlastfoot
\hline
1 & $ - {\tau} - {\tau}^{2} - {\tau}^{3}$ & 
36 & 
$(-1, -1, -1, 0)_{\tau}$ & 4 \\\hline
2 & $ + {\tau} + {\tau}^{2} + {\tau}^{3}$ & 
36 & 
$(1, 1, 1, 0)_{\tau}$ & 4 \\\hline
3 & $1 - {\tau} - {\tau}^{2} - {\tau}^{3}$ & 
29 & 
$(-1, 0, 0, 1, -3, -3)_{\tau}$ & 6 \\\hline
4 & $-1 + {\tau} + {\tau}^{2} + {\tau}^{3}$ & 
29 & 
$(1, 0, 0, -1, 3, 3)_{\tau}$ & 6 \\\hline
5 & $2 - {\tau} - {\tau}^{2} - {\tau}^{3}$ & 
26 & 
$(-1, 0, 0, 1, -3, -2)_{\tau}$ & 6 \\\hline
6 & $-2 + {\tau} + {\tau}^{2} + {\tau}^{3}$ & 
26 & 
$(1, 0, 0, -1, 3, 2)_{\tau}$ & 6 \\\hline
7 & $3 - {\tau} - {\tau}^{2} - {\tau}^{3}$ & 
27 & 
$(-1, 0, 0, 1, -3, -1)_{\tau}$ & 6 \\\hline
8 & $-3 + {\tau} + {\tau}^{2} + {\tau}^{3}$ & 
27 & 
$(1, 0, 0, -1, 3, 1)_{\tau}$ & 6 \\
\end{longtable}
\end{center}

\begin{center}
\begin{longtable}[c]{|c||c|c|c|c|}
\caption{$\alpha^{\prime} \in \mathbb{Z}[{\tau}]$ %
with $\mathcal{N}(\alpha^{\prime})^{2} \leq 38$ ($\mu = -1$). }
\\
\hline
\label{table:hyperKC_tnaf_existence_a=0}
${\#}$ & ${\alpha^{\prime}} = s + t{\tau} + u{\tau}^{2} + v{\tau}^{3}$ & %
$\mathcal{N}(\alpha^{\prime})^{2}$ & GLS ${\tau}$-adic expansion of $\alpha^{\prime}$ & $\ell(\alpha^{\prime})$ \\
\hline\hline
\endfirsthead
\hline
${\#}$ & ${\alpha^{\prime}} = s + t{\tau} + u{\tau}^{2} + v{\tau}^{3}$ & %
$\mathcal{N}(\alpha^{\prime})^{2}$ & GLS ${\tau}$-adic expansion of $\alpha^{\prime}$ & $\ell(\alpha^{\prime})$ \\
\hline\hline
\endhead
\hline
\endfoot
\hline
\endlastfoot
\hline
1 & $-2 - {\tau} - 2{\tau}^{2} - {\tau}^{3}$ & 
38 & 
$(1, 0, 0, 2, 0, -3, 2)_{\tau}$ & 7 \\\hline
2 & $2 + {\tau} + 2{\tau}^{2} + {\tau}^{3}$ & 
38 & 
$(-1, 0, 0, -2, 0, 3, -2)_{\tau}$ & 7 \\\hline
3 & $-1 - {\tau} - 2{\tau}^{2} - {\tau}^{3}$ & 
38 & 
$(1, 0, 0, 2, 0, -3, 3)_{\tau}$ & 7 \\\hline
4 & $1 + {\tau} + 2{\tau}^{2} + {\tau}^{3}$ & 
38 & 
$(-1, 0, 0, -2, 0, 3, -3)_{\tau}$ & 7 \\\hline
5 & $-2 - 2{\tau}^{2} - {\tau}^{3}$ & 
38 & 
$(-1, -2, 0, -2)_{\tau}$ & 4 \\\hline
6 & $2 + 2{\tau}^{2} + {\tau}^{3}$ & 
38 & 
$(1, 2, 0, 2)_{\tau}$ & 4 \\\hline
7 & $-1 - 2{\tau}^{2} - {\tau}^{3}$ & 
37 & 
$(-1, -2, 0, -1)_{\tau}$ & 4 \\\hline
8 & $1 + 2{\tau}^{2} + {\tau}^{3}$ & 
37 & 
$(1, 2, 0, 1)_{\tau}$ & 4 \\
\end{longtable}
\end{center}

From Table~\ref{table:hyperKC_tnaf_existence_a=1} 
(or Table~\ref{table:hyperKC_tnaf_existence_a=0}), 
there are 300 elements whose square values of norm are less than or equal to $38$. 
The GLS ${\tau}$-adic expansion of ${\alpha}^{\prime}$ with 
$\mathcal{N}({\alpha}^{\prime}) \leq 38$ are also shown 
in Table~\ref{table:hyperKC_tnaf_existence_a=1} 
and Table~\ref{table:hyperKC_tnaf_existence_a=0}. 
This concludes the proof. 
\end{proof}

\begin{thm} 
\label{thm:hyperKC_GLS_non_unique}
{\rm\bfseries 
[Non-uniqueness of the GLS ${\tau}$-adic expansion]}\ 
It does not hold the uniqueness for GLS ${\tau}$-adic expansion %
with respect to $\mathcal{D}$ on hyperelliptic Koblitz curves. 
\end{thm}

\begin{proof}
We give an element which has two different GLS ${\tau}$-adic expansion. 
We set $\alpha = b{\tau}^{2} + 2\mu{\tau} - 1 \in \mathbb{Z}[\tau]$ 
($b \in \mathcal{D}$). 
Since $b \in \mathcal{D}$, 
$\alpha = (0, b, 2\mu, -1)_{\tau}$ is a GLS ${\tau}$-adic expansion. 
On the other hand, by Algorithm~\ref{alg:hyperKC_GLS}, 
we also have $\alpha = (1, -\mu, b, 0, 3)_{\tau}$, 
which is another GLS ${\tau}$-adic expansion. 
Hence there exists an element which can not be uniquely represented 
as a GLS ${\tau}$-adic expansion. 
\end{proof}

\begin{rem}
Indeed, for $\mu = \pm{1}$ there are $252$ elements 
of the form $\alpha = ({c_3}, {c_2}, {c_1}, {c_0})_{\tau}$, 
where ${c_i} \in \mathcal{D}$ and ${c_{3}}{c_{2}}{c_{1}}{c_{0}} = 0$. 
These elements are listed in Table~\ref{table:hyperKC_non_gls_example_1} 
and Table~\ref{table:hyperKC_non_gls_example_2} 
in Appendix~\ref{sec:hyperKC_append_gls_non_uniqueness}. 
\end{rem}

\begin{thm} 
\label{thm:hyperKC_GLS_non_optimality}
{\rm\bfseries 
[Non-minimality of GLS ${\tau}$-adic expansion]}\ 
It does not hold the minimality for GLS ${\tau}$-adic expansion %
with respect to $\mathcal{D}$ on hyperelliptic Koblitz curves. 
\end{thm}

\begin{proof}
We give an example that is an element $\alpha \in \mathbb{Z}[\tau]$ 
which does not have minimal Hamming weight. 
Let us take $\alpha = b{\tau}^{2} + 2\mu{\tau} - 1 \in \mathbb{Z}[\tau]$ again. 
From $\alpha = (0, b, 2\mu, -1)_{\tau} = (1, -\mu, b, 0, 3)_{\tau}$, 
GLS ${\tau}$-adic expansion does not have minimal Hamming weight 
among all Frobenius expansion with digit set $\mathcal{D}$. 
\end{proof}


\section{Properties of $\tau$-NAF}
\label{sec:hyperKC_prop_tnaf}

In this section, we investigate the three properties of 
$\tau$-NAF on hyperelliptic Koblitz curves. 
In \cite[p.19]{GLS00}, \cite[p.115]{GLS01}, 
the authors proposed to use the digit set 
$\widetilde{\mathcal{D}} = %
\{ 0, \pm{1}, \pm{2}, \pm{(1 + {\tau})}, %
\pm{(1 - {\tau})}, \pm{(1 - 2{\tau})}, \pm{2} + {\tau} \}$ 
to construct $\tau$-NAF. 
However, it seems that 
$\widetilde{\mathcal{D}}$ is a digit set specifically for the case $a = 1$. 
We start with determining possible digit sets so that 
one can achieve sparse $\tau$-adic expansions. 

\subsection{Possible digit sets}
\label{subsec:hyperKC_tnaf_digitset} 

Let $\alpha = s + t{\tau} + u{\tau}^{2} %
+ v{\tau}^{3} \in {\mathbb{Z}[\tau]}$ 
($s, t, u, v \in \mathbb{Z}$). 
According to the discussion 
in \cite[pp.7--8]{GLS00}, \cite[p.111]{GLS01}, 
we put $s = 8{Q_s} + {R_s}$ ($0 \leq {R_s} \leq 7$) 
and $t = 4{Q_t} + {R_t}$ ($0 \leq {R_t} \leq 3$). 

In order to realize a sparce ${\tau}$-adic expansion, 
we would like to find $c \in \mathbb{Z}[\tau]$ 
such that ${\tau} \mid (\alpha - c)$ when $\tau \mid \alpha$ and 
${\tau}^{2} \mid (\alpha - c)$ when $\tau \centernot\mid \alpha$. 
As in the case of the GLS ${\tau}$-adic expansions, 
if $\tau \mid \alpha$ then we choose $c = 0$. 
Suppose ${\tau} \centernot\mid {\alpha}$. 
We first derive two necessity conditions on ${R_s}$ and ${R_t}$ 
so that ${\alpha} - c$ is divisible by ${\tau}^{2}$. 
From ${\tau} \centernot\mid {\alpha}$, 
we have $1 \leq {R_s} \leq 7$, ${R_s} \neq 4$ 
and $1 \leq {R_t} \leq 3$. 
By Lemma~\ref{lem:hyperKC_div}, 
$({R_s} - c^{\prime})$ is divisible by $4$. 

\begin{description}
\item[(Condition~1)] 
$({R_s} - c^{\prime})$ is divisible by $4$. 
\end{description}

We denote $({R_s} - c^{\prime})/4$ by $\tilde{c}$. 
Then 
\begin{eqnarray*}
\alpha - c 
& = & 
{R_s} + (4({\mu}{Q_s} + {Q_t}) + {R_t}){\tau} %
+ u{\tau}^{2} + (v + 2{\mu}{Q_s}){\tau}^{3} %
- 2{Q_s}{\tau}^{4} - c \nonumber \\%
& = &
({R_s} - c^{\prime}) + (4({\mu}{Q_s} + {Q_t}) + ({R_t} - c^{\prime\prime})){\tau} %
+ u{\tau}^{2} + (v + 2{\mu}{Q_s}){\tau}^{3} %
- 2{Q_s}{\tau}^{4} \nonumber \\%
& = &
2{\mu}\tilde{c}{\tau} + {\mu}\tilde{c}{\tau}^{3} - \tilde{c}{\tau}^{4} %
+ (4({\mu}{Q_s} + {Q_t}) + ({R_t} - c^{\prime\prime})){\tau} \\%
&&\hspace{12em} %
+ u{\tau}^{2} + (v + 2{\mu}{Q_s}){\tau}^{3} %
- 2{Q_s}{\tau}^{4} \nonumber \\%
& = &
(4({\mu}{Q_s} + {Q_t}) + ({R_t} - c^{\prime\prime} + 2{\mu}\tilde{c})){\tau} \\%
&&\hspace{12em} %
+ u{\tau}^{2} + (v + 2{\mu}{Q_s} + {\mu}\tilde{c}){\tau}^{3} %
- (2{Q_s} + \tilde{c}){\tau}^{4}. %
\end{eqnarray*}
Now once again from Lemma~\ref{lem:hyperKC_div}, it follows that 
$({R_t} - c^{\prime\prime} + 2{\mu}\tilde{c})$ is divided by $4$. 

\begin{description}
\item[(Condition~2)] 
$({R_t} - c^{\prime\prime} + 2{\mu}\tilde{c})$ is divisible by $4$. 
\end{description}

Remark that each element $c \in \widetilde{\mathcal{D}}$ 
has the form $c^{\prime} + c^{\prime\prime}{\tau}$, 
where $-2 \leq c^{\prime}, c^{\prime\prime} \leq 2$ 
and $(|c^{\prime}|, |c^{\prime\prime}|) \neq (2, 2)$. 
So we assume that $c = c^{\prime} + c^{\prime\prime}{\tau} \in \mathbb{Z}[\tau]$ 
is satisfied the following condition: 

\begin{description}
\item[(Condition~3)] 
$-2 \leq c^{\prime}, c^{\prime\prime} \leq 2$ 
and $(|c^{\prime}|, |c^{\prime\prime}|) \neq (2, 2)$. 
\end{description}

The table below provides how to choose 
$c = c^{\prime} + c^{\prime\prime}{\tau}$ 
so that $(\alpha - c)$ is divisible by ${\tau}^{2}$. 

\begin{table}[htbp]
\begin{center}
\begin{tabular}{|c|c|c|c|c|c|c|}\hline 
\backslashbox{${R_t}$}{${R_s}$} & %
$1$ & $2$ & $3$ & $5$ & $6$ & $7$ \\\hline\hline
 & %
 &  & $-1 + 2{\mu}{\tau}$ & $1 + 2{\mu}{\tau}$ &  &  \\
$0$ & %
$1$ & $2$ & or & or & $-2$ & $-1$ \\
 & %
 &  & $-1 - 2{\mu}{\tau}$ & $1 - 2{\mu}{\tau}$ &  &  \\
\hline
 & %
 & $2 + {\tau}$ &  &  & $-2 + {\tau}$ &  \\
$1$ & %
$1 + {\tau}$ & or & $-1 - {\tau}$ & $1 - {\tau}$ & or & $-1 + {\tau}$ \\
 & %
 & $-2 - {\tau}$ &  &  & $2 - {\tau}$ &  \\
\hline
 & %
$1 + 2{\tau}$ &  &  &  &  & $-1 + 2{\tau}$ \\
$2$ & %
or & $-2$ & $-1$ & $1$ & $2$ & or \\
 & %
$1 - 2{\tau}$ &  &  &  &  & $-1 - 2{\tau}$ \\
\hline
 & %
 & $-2 + {\tau}$ &  &  & $2 + {\tau}$ &  \\
$3$ & %
$1 - {\tau}$ & or & $-1 + {\tau}$ & $1 + {\tau}$ & or & $-1 - {\tau}$ \\
 & %
 & $2 - {\tau}$ &  &  & $-2 - {\tau}$ &  \\
\hline
\end{tabular}
\end{center}
\end{table}

Here, in order to avoid a computational overhead 
in the precomputation for scalar multiplication, 
we impose the following additional condition: 

\begin{description}
\item[(Condition~4)] 
In cases that there exist two candidates 
$c = c^{\prime} + c^{\prime\prime}{\tau}$ such that 
$\tau \mid (\alpha - c)$ for two different pairs of $({R_s}, {R_t})$, 
only one element of the two candidates is included in a digit set. 
\end{description}

We are now in a position to clarify the digit sets so that 
one may achieve sparse ${\tau}$-adic expansions. 
We put 
\[
\widetilde{\mathcal{D}_{0}} = \{ 0, \pm{1}, \pm{2}, 1 \pm{\tau}, -1 \pm{\tau} \}, 
\]
\[
\widetilde{\mathcal{D}_{1}^{j_1}} = \bigl{\{} (-1)^{\frac{{j_1} - {j_1} \bmod 2}{2}}(2 + {\tau}), %
(-1)^{j_1}(2 - {\tau}) \bigr{\}} \ (0 \leq {j_1} \leq 3), 
\]
\[
\widetilde{\mathcal{D}_{2}^{j_2}} = \bigl{\{} 1 + 2{\mu}(-1)^{\frac{{j_2} - {j_2} \bmod 2}{2}}{\tau}, -1 + 2{\mu}(-1)^{j_2}{\tau} \bigr{\}} \ (0 \leq {j_2} \leq 3), 
\]
respectively. 
Then 
we obtain the following $16$ possible digit sets: 
\[
\widetilde{\mathcal{D}_{j}} := \widetilde{\mathcal{D}_{0}} %
\cup \widetilde{\mathcal{D}_{1}^{j_1}} %
\cup \widetilde{\mathcal{D}_{2}^{j_2}} %
\quad (j := 4{j_1} + {j_2} + 1, 0 \leq {j_1}, {j_2} \leq 3). 
\]
Remark that 
$\widetilde{\mathcal{D}_{7}} = \widetilde{\mathcal{D}}$ 
if $\mu = 1$. 
Our procedure to construct ${\tau}$-NAF 
is same as \cite{GLS00}, \cite{GLS01}. 
First we fix a digit set $\widetilde{\mathcal{D}_{j}}$. 
Set ${\alpha}_{i} := {\alpha} = s + t{\tau} + u{\tau}^{2} %
+ v{\tau}^{3} \in {\mathbb{Z}[\tau]}$ 
($s, t, u, v \in \mathbb{Z}$). The initial index $i$ is zero. 
We choose a $c$ according to the above table and 
${\alpha}_{i+1} := ({\alpha}_{i} - c)/{\tau}$, $i:=i + 1$. 

Repeating the process until ${\alpha}_{i}$ will be zero for some $i$ 
may yields a sparse ${\tau}$-adic expansion. 

\subsection{Properties}
\label{subsec:hyperKC_tnaf_properties} 

Indeed, we can prove that the process will terminate 
after a finite number of iterations. 
Namely, one can construct $\tau$-NAF for each digit set 
$\widetilde{\mathcal{D}_{j}}$ ($j = 1, \ldots, 16$). 

\begin{thm} 
\label{thm:hyperKC_tnaf_existence}
{\rm\bfseries 
[Existence of ${\tau}$-NAF]}\ 
For each $j$ $(j = 1, \ldots, 16)$, 
every $\alpha \in \mathbb{Z}[\tau]$ has 
a ${\tau}$-NAF with respect to 
the digit set $\widetilde{\mathcal{D}_{j}}$. 
\end{thm}

\begin{proof}
We keep $j$ ($1 \leq j \leq 16$) fixed and 
consider the digit set $\widetilde{\mathcal{D}_{j}}$. 
Let us take ${\alpha}_{0} := {\alpha} = s + t{\tau} + u{\tau}^{2} %
+ v{\tau}^{3} \in {\mathbb{Z}[\tau]}$ 
($s, t, u, v \in \mathbb{Z}$). 
First, we claim that there exist 
${\alpha}^{\prime} \in {\mathbb{Z}[\tau]}$, 
${\ell}^{\prime} \in \mathbb{Z}_{>0}$ and 
$c_{0}, c_{1}, \ldots, c_{{\ell}^{\prime} - 1} \in \widetilde{\mathcal{D}_{j}}$ 
such that $\mathcal{N}({\alpha}^{\prime}) \leq 2\sqrt{5}$ and 
\[
\alpha = \sum_{j=0}^{{\ell}^{\prime}-1} {c_{j}} + {\alpha}^{\prime}{\tau}^{{\ell}^{\prime}}. %
\]
Notice that 
$\max \{ \mathcal{N}(\alpha) \mid \alpha %
\in \widetilde{\mathcal{D}_{j}}, %
j = 1, \ldots, 16 \} = 2\sqrt{5}$ 
(the equality holds when 
$\alpha = \pm{1} + 2{\mu}{\tau}$). 
We put ${\alpha}_{i + 1} := ({\alpha}_{i} - {c_i})/{\tau}$ 
(${c_i} \in \widetilde{\mathcal{D}_{j}}$). 
If ${c_i} = 0$, it always satisfies that 
$\mathcal{N}({\alpha}_{i + 1}) \leq \mathcal{N}({\alpha}_{i})/\sqrt{2} %
< \mathcal{N}({\alpha}_{i})$ for ${\alpha}_{i} \neq 0$. 
Otherwise, a simple calculation shows 
$\mathcal{N}({\alpha}_{i + 2}) < \mathcal{N}({\alpha}_{i})$ 
for $\mathcal{N}({\alpha}_{i}) > 2\sqrt{5}$ 
because $\mathcal{N}({\alpha}_{i + 2}) \leq %
\mathcal{N}({\alpha}_{i})/2 + \sqrt{5} < \mathcal{N}({\alpha}_{i})$. 
Thus for ${\alpha}_{i} \in \mathbb{Z}[\tau]$ 
with $\mathcal{N}({\alpha}_{i}) > 2\sqrt{5}$, 
there exists ${j_0} \in \{1, 2 \}$ such that 
$\mathcal{N}({\alpha}_{i + {j_0}}) < \mathcal{N}({\alpha}_{i})$. 
Remark that for any $\alpha \in \mathbb{Z}[\tau]$, we have 
$\mathcal{N}({\alpha})^{2} \in \mathbb{Z}$. 
Then there exists ${i_0} \in \mathbb{Z}$ such that 
$\mathcal{N}({\alpha}_{i_0}) < 2\sqrt{5}$. 
By taking ${\ell}^{\prime} \in \mathbb{Z}$ 
be the minimal such ${i_0}$ 
and ${\alpha}^{\prime} := {\alpha}_{{\ell}^{\prime}}$, 
the first assertion follows. 

The only remaining issues are to find all elements 
${\alpha}^{\prime} \in {\mathbb{Z}[\tau]}$ 
satisfying $\mathcal{N}({\alpha}^{\prime}) \leq 2\sqrt{5}$, 
namely $\mathcal{N}({\alpha}^{\prime})^{2} \leq 20$ 
and to show the finiteness of the ${\tau}$-NAF for each element 
which is found by the search. 
As in the proof of Theorem~\ref{thm:hyperKC_GLS_existence}, 
we use the Fincke-Pohst algorithm to find them. 

\begin{center}
\begin{longtable}[c]{|c||c|c|c|c|}
\caption{Elements $\alpha^{\prime} \in \mathbb{Z}[{\tau}]$ with $\mathcal{N}(\alpha^{\prime})^{2} \leq 20$ 
and their ${\tau}$-NAF ($\widetilde{\mathcal{D}_{1}}$, $\mu = 1$). }
\\
\hline
\label{table:hyperKC_tnaf_existence_D1_mu=1}
${\#}$ & ${\alpha^{\prime}} = s + t{\tau} + u{\tau}^{2} + v{\tau}^{3}$ & %
$\mathcal{N}(\alpha^{\prime})^{2}$ & ${\tau}$-NAF of $\alpha^{\prime}$ & $\ell(\alpha^{\prime})$ \\
\hline\hline
\endfirsthead
\hline
${\#}$ & ${\alpha^{\prime}} = s + t{\tau} + u{\tau}^{2} + v{\tau}^{3}$ & %
$\mathcal{N}(\alpha^{\prime})^{2}$ & ${\tau}$-NAF of $\alpha^{\prime}$ & $\ell(\alpha^{\prime})$ \\
\hline\hline
\endhead
\hline
\endfoot
\hline
\endlastfoot
\hline
1 & $2 - {\tau}^{2} - {\tau}^{3}$ & 
20 & 
$(-1-\tau, 0, 2)_{\tau}$ & 3 \\\hline
2 & $-2 + {\tau}^{2} + {\tau}^{3}$ & 
20 & 
$(1+\tau, 0, -2)_{\tau}$ & 3 \\\hline
3 & $1 - {\tau} - {\tau}^{3}$ & 
16 & 
$(-1, 0, 0, 1-\tau)_{\tau}$ & 4 \\\hline
4 & $-1 + {\tau} + {\tau}^{3}$ & 
16 & 
$(1, 0, 0, -1+\tau)_{\tau}$ & 4 \\\hline
5 & $2 - {\tau} - {\tau}^{3}$ & 
14 & 
$(-1, 0, 0, 2-\tau)_{\tau}$ & 4 \\\hline
6 & $-2 + {\tau} + {\tau}^{3}$ & 
14 & 
$(1, 0, 0, 0, 2-\tau)_{\tau}$ & 5 \\\hline
7 & $3 - {\tau} - {\tau}^{3}$ & 
16 & 
$(-1, 0, 0, 0, -1+\tau)_{\tau}$ & 5 \\\hline
8 & $-3 + {\tau} + {\tau}^{3}$ & 
16 & 
$(1, 0, 0, 0, 1-\tau)_{\tau}$ & 5 \\\hline
9 & $ - {\tau}^{3}$ & 
16 & 
$(-1, 0, 0, 0)_{\tau}$ & 4 \\\hline
10 & $ + {\tau}^{3}$ & 
16 & 
$(1, 0, 0, 0)_{\tau}$ & 4 \\\hline
11 & $1 - {\tau}^{3}$ & 
11 & 
$(-1, 0, 0, 1)_{\tau}$ & 4 \\\hline
12 & $-1 + {\tau}^{3}$ & 
11 & 
$(1, 0, 0, -1)_{\tau}$ & 4 \\\hline
13 & $2 - {\tau}^{3}$ & 
10 & 
$(-1, 0, 0, 2)_{\tau}$ & 4 \\\hline
14 & $-2 + {\tau}^{3}$ & 
10 & 
$(1, 0, 0, -2)_{\tau}$ & 4 \\\hline
15 & $3 - {\tau}^{3}$ & 
13 & 
$(-1, 0, 0, 0, -1+2\tau)_{\tau}$ & 5 \\\hline
16 & $-3 + {\tau}^{3}$ & 
13 & 
$(1, 0, 0, -2, 0, 1+2\tau)_{\tau}$ & 6 \\\hline
17 & $4 - {\tau}^{3}$ & 
20 & 
$(-1, 0, 0, 2, 0)_{\tau}$ & 5 \\\hline
18 & $-4 + {\tau}^{3}$ & 
20 & 
$(1, 0, 0, -2, 0)_{\tau}$ & 5 \\\hline
19 & $ + {\tau} - {\tau}^{3}$ & 
18 & 
$(-1, 0, 1, 0)_{\tau}$ & 4 \\\hline
20 & $ - {\tau} + {\tau}^{3}$ & 
18 & 
$(1, 0, -1, 0)_{\tau}$ & 4 \\\hline
21 & $1 + {\tau} - {\tau}^{3}$ & 
14 & 
$(-1, 0, 0, 1+\tau)_{\tau}$ & 4 \\\hline
22 & $-1 - {\tau} + {\tau}^{3}$ & 
14 & 
$(1, 0, 0, -1-\tau)_{\tau}$ & 4 \\\hline
23 & $2 + {\tau} - {\tau}^{3}$ & 
14 & 
$(-1, 0, 0, 2+\tau)_{\tau}$ & 4 \\\hline
24 & $-2 - {\tau} + {\tau}^{3}$ & 
14 & 
$(1, 0, 0, -2, 0, 2+\tau)_{\tau}$ & 6 \\\hline
25 & $3 + {\tau} - {\tau}^{3}$ & 
18 & 
$(-1, 0, 0, 2, 0, -1-\tau)_{\tau}$ & 6 \\\hline
26 & $-3 - {\tau} + {\tau}^{3}$ & 
18 & 
$(1, 0, 0, -2, 0, 1+\tau)_{\tau}$ & 6 \\\hline
27 & $1 - {\tau} + {\tau}^{2} - {\tau}^{3}$ & 
19 & 
$(1-\tau, 0, 1-\tau)_{\tau}$ & 3 \\\hline
28 & $-1 + {\tau} - {\tau}^{2} + {\tau}^{3}$ & 
19 & 
$(-1+\tau, 0, -1+\tau)_{\tau}$ & 3 \\\hline
29 & $2 - {\tau} + {\tau}^{2} - {\tau}^{3}$ & 
18 & 
$(1-\tau, 0, 2-\tau)_{\tau}$ & 3 \\\hline
30 & $-2 + {\tau} - {\tau}^{2} + {\tau}^{3}$ & 
18 & 
$(1, 0, -1, 0, 2-\tau)_{\tau}$ & 5 \\\hline
31 & $ + {\tau}^{2} - {\tau}^{3}$ & 
20 & 
$(1-\tau, 0, 0)_{\tau}$ & 3 \\\hline
32 & $ - {\tau}^{2} + {\tau}^{3}$ & 
20 & 
$(-1+\tau, 0, 0)_{\tau}$ & 3 \\\hline
33 & $1 + {\tau}^{2} - {\tau}^{3}$ & 
16 & 
$(1-\tau, 0, 1)_{\tau}$ & 3 \\\hline
34 & $-1 - {\tau}^{2} + {\tau}^{3}$ & 
16 & 
$(-1+\tau, 0, -1)_{\tau}$ & 3 \\\hline
35 & $2 + {\tau}^{2} - {\tau}^{3}$ & 
16 & 
$(1-\tau, 0, 2)_{\tau}$ & 3 \\\hline
36 & $-2 - {\tau}^{2} + {\tau}^{3}$ & 
16 & 
$(-1+\tau, 0, -2)_{\tau}$ & 3 \\\hline
37 & $3 + {\tau}^{2} - {\tau}^{3}$ & 
20 & 
$(-1, 0, 1, 0, -1+2\tau)_{\tau}$ & 5 \\\hline
38 & $-3 - {\tau}^{2} + {\tau}^{3}$ & 
20 & 
$(1, 0, 0, -2, 0, 1+2\tau, 0, 1+2\tau)_{\tau}$ & 8 \\\hline
39 & $-1 - {\tau} - {\tau}^{2}$ & 
18 & 
$(-1, 0, -1-\tau)_{\tau}$ & 3 \\\hline
40 & $1 + {\tau} + {\tau}^{2}$ & 
18 & 
$(1, 0, 1+\tau)_{\tau}$ & 3 \\\hline
41 & $ - {\tau} - {\tau}^{2}$ & 
14 & 
$(-1-\tau, 0)_{\tau}$ & 2 \\\hline
42 & $ + {\tau} + {\tau}^{2}$ & 
14 & 
$(1+\tau, 0)_{\tau}$ & 2 \\\hline
43 & $1 - {\tau} - {\tau}^{2}$ & 
14 & 
$(-1, 0, 1-\tau)_{\tau}$ & 3 \\\hline
44 & $-1 + {\tau} + {\tau}^{2}$ & 
14 & 
$(1, 0, -1+\tau)_{\tau}$ & 3 \\\hline
45 & $2 - {\tau} - {\tau}^{2}$ & 
18 & 
$(-1, 0, 2-\tau)_{\tau}$ & 3 \\\hline
46 & $-2 + {\tau} + {\tau}^{2}$ & 
18 & 
$(1, 0, 1-\tau, 0, 2-\tau)_{\tau}$ & 5 \\\hline
47 & $-2 - {\tau}^{2}$ & 
18 & 
$(-1, 0, -2)_{\tau}$ & 3 \\\hline
48 & $2 + {\tau}^{2}$ & 
18 & 
$(1, 0, 2)_{\tau}$ & 3 \\\hline
49 & $-1 - {\tau}^{2}$ & 
11 & 
$(-1, 0, -1)_{\tau}$ & 3 \\\hline
50 & $1 + {\tau}^{2}$ & 
11 & 
$(1, 0, 1)_{\tau}$ & 3 \\\hline
51 & $ - {\tau}^{2}$ & 
8 & 
$(-1, 0, 0)_{\tau}$ & 3 \\\hline
52 & $ + {\tau}^{2}$ & 
8 & 
$(1, 0, 0)_{\tau}$ & 3 \\\hline
53 & $1 - {\tau}^{2}$ & 
9 & 
$(-1, 0, 1)_{\tau}$ & 3 \\\hline
54 & $-1 + {\tau}^{2}$ & 
9 & 
$(1, 0, -1)_{\tau}$ & 3 \\\hline
55 & $2 - {\tau}^{2}$ & 
14 & 
$(-1, 0, 2)_{\tau}$ & 3 \\\hline
56 & $-2 + {\tau}^{2}$ & 
14 & 
$(1, 0, -2)_{\tau}$ & 3 \\\hline
57 & $-2 + {\tau} - {\tau}^{2}$ & 
18 & 
$(1, 0, -1-\tau, 0, 2-\tau)_{\tau}$ & 5 \\\hline
58 & $2 - {\tau} + {\tau}^{2}$ & 
18 & 
$(1, 0, 2-\tau)_{\tau}$ & 3 \\\hline
59 & $-1 + {\tau} - {\tau}^{2}$ & 
12 & 
$(-1, 0, -1+\tau)_{\tau}$ & 3 \\\hline
60 & $1 - {\tau} + {\tau}^{2}$ & 
12 & 
$(1, 0, 1-\tau)_{\tau}$ & 3 \\\hline
61 & $ + {\tau} - {\tau}^{2}$ & 
10 & 
$(1-\tau, 0)_{\tau}$ & 2 \\\hline
62 & $ - {\tau} + {\tau}^{2}$ & 
10 & 
$(-1+\tau, 0)_{\tau}$ & 2 \\\hline
63 & $1 + {\tau} - {\tau}^{2}$ & 
12 & 
$(-1, 0, 1+\tau)_{\tau}$ & 3 \\\hline
64 & $-1 - {\tau} + {\tau}^{2}$ & 
12 & 
$(1, 0, -1-\tau)_{\tau}$ & 3 \\\hline
65 & $2 + {\tau} - {\tau}^{2}$ & 
18 & 
$(-1, 0, 2+\tau)_{\tau}$ & 3 \\\hline
66 & $-2 - {\tau} + {\tau}^{2}$ & 
18 & 
$(1, 0, 0, -1-\tau, 0, 2+\tau)_{\tau}$ & 6 \\\hline
67 & $ + 2{\tau} - {\tau}^{2}$ & 
20 & 
$(2-\tau, 0)_{\tau}$ & 2 \\\hline
68 & $ - 2{\tau} + {\tau}^{2}$ & 
20 & 
$(-1+\tau, 0, 0, 2-\tau, 0)_{\tau}$ & 5 \\\hline
69 & $-1 - 2{\tau}$ & 
20 & 
$(-1+\tau, 0, -2, 0, -1+2\tau)_{\tau}$ & 5 \\\hline
70 & $1 + 2{\tau}$ & 
20 & 
$(1+2\tau)_{\tau}$ & 1 \\\hline
71 & $ - 2{\tau}$ & 
16 & 
$(-2, 0)_{\tau}$ & 2 \\\hline
72 & $ + 2{\tau}$ & 
16 & 
$(2, 0)_{\tau}$ & 2 \\\hline
73 & $1 - 2{\tau}$ & 
16 & 
$(-1+\tau, 0, -2, 0, 1+2\tau)_{\tau}$ & 5 \\\hline
74 & $-1 + 2{\tau}$ & 
16 & 
$(-1+2\tau)_{\tau}$ & 1 \\\hline
75 & $2 - 2{\tau}$ & 
20 & 
$(1-\tau, 0, 0, -2)_{\tau}$ & 4 \\\hline
76 & $-2 + 2{\tau}$ & 
20 & 
$(-1+\tau, 0, 0, 2)_{\tau}$ & 4 \\\hline
77 & $-2 - {\tau}$ & 
14 & 
$(1, 0, 0, -2, 0, 2+\tau, 0, 2+\tau)_{\tau}$ & 8 \\\hline
78 & $2 + {\tau}$ & 
14 & 
$(2+\tau)_{\tau}$ & 1 \\\hline
79 & $-1 - {\tau}$ & 
7 & 
$(-1-\tau)_{\tau}$ & 1 \\\hline
80 & $1 + {\tau}$ & 
7 & 
$(1+\tau)_{\tau}$ & 1 \\\hline
81 & $ - {\tau}$ & 
4 & 
$(-1, 0)_{\tau}$ & 2 \\\hline
82 & $ + {\tau}$ & 
4 & 
$(1, 0)_{\tau}$ & 2 \\\hline
83 & $1 - {\tau}$ & 
5 & 
$(1-\tau)_{\tau}$ & 1 \\\hline
84 & $-1 + {\tau}$ & 
5 & 
$(-1+\tau)_{\tau}$ & 1 \\\hline
85 & $2 - {\tau}$ & 
10 & 
$(2-\tau)_{\tau}$ & 1 \\\hline
86 & $-2 + {\tau}$ & 
10 & 
$(-1+\tau, 0, 0, 2-\tau)_{\tau}$ & 4 \\\hline
87 & $3 - {\tau}$ & 
19 & 
$(1-\tau, 0, 0, -1+\tau)_{\tau}$ & 4 \\\hline
88 & $-3 + {\tau}$ & 
19 & 
$(-1+\tau, 0, 0, 1-\tau)_{\tau}$ & 4 \\\hline
89 & $-3$ & 
18 & 
$(1, 0, 0, -2, 0, 2+\tau, 0, 1+2\tau)_{\tau}$ & 8 \\\hline
90 & $3$ & 
18 & 
$(1-\tau, 0, 0, -1+2\tau)_{\tau}$ & 4 \\\hline
91 & $-2$ & 
8 & 
$(-2)_{\tau}$ & 1 \\\hline
92 & $2$ & 
8 & 
$(2)_{\tau}$ & 1 \\\hline
93 & $-1$ & 
2 & 
$(-1)_{\tau}$ & 1 \\\hline
94 & $1$ & 
2 & 
$(1)_{\tau}$ & 1 \\\hline
\end{longtable}
\end{center}

\begin{center}
\begin{longtable}[c]{|c||c|c|c|c|}
\caption{Elements $\alpha^{\prime} \in \mathbb{Z}[{\tau}]$ with $\mathcal{N}(\alpha^{\prime})^{2} \leq 20$ 
and their ${\tau}$-NAF ($\widetilde{\mathcal{D}_{1}}$, $\mu = -1$). }
\\
\hline
\label{table:hyperKC_tnaf_existence_D1_mu=-1}
${\#}$ & ${\alpha^{\prime}} = s + t{\tau} + u{\tau}^{2} + v{\tau}^{3}$ & %
$\mathcal{N}(\alpha^{\prime})^{2}$ & ${\tau}$-NAF of $\alpha^{\prime}$ & $\ell(\alpha^{\prime})$ \\
\hline\hline
\endfirsthead
\hline
${\#}$ & ${\alpha^{\prime}} = s + t{\tau} + u{\tau}^{2} + v{\tau}^{3}$ & %
$\mathcal{N}(\alpha^{\prime})^{2}$ & ${\tau}$-NAF of $\alpha^{\prime}$ & $\ell(\alpha^{\prime})$ \\
\hline\hline
\endhead
\hline
\endfoot
\hline
\endlastfoot
\hline
1 & $-2 - {\tau} - {\tau}^{2} - {\tau}^{3}$ & 
18 & 
$(1, 0, -1, 0, 2+\tau)_{\tau}$ & 5 \\\hline
2 & $2 + {\tau} + {\tau}^{2} + {\tau}^{3}$ & 
18 & 
$(1+\tau, 0, 2+\tau)_{\tau}$ & 3 \\\hline
3 & $-1 - {\tau} - {\tau}^{2} - {\tau}^{3}$ & 
19 & 
$(-1-\tau, 0, -1-\tau)_{\tau}$ & 3 \\\hline
4 & $1 + {\tau} + {\tau}^{2} + {\tau}^{3}$ & 
19 & 
$(1+\tau, 0, 1+\tau)_{\tau}$ & 3 \\\hline
5 & $-3 - {\tau}^{2} - {\tau}^{3}$ & 
20 & 
$(-1, 0, 0, -2, 0, 1-2\tau, 0, 1-2\tau)_{\tau}$ & 8 \\\hline
6 & $3 + {\tau}^{2} + {\tau}^{3}$ & 
20 & 
$(-1, 0, 1, 0, -1-2\tau)_{\tau}$ & 5 \\\hline
7 & $-2 - {\tau}^{2} - {\tau}^{3}$ & 
16 & 
$(-1-\tau, 0, -2)_{\tau}$ & 3 \\\hline
8 & $2 + {\tau}^{2} + {\tau}^{3}$ & 
16 & 
$(1+\tau, 0, 2)_{\tau}$ & 3 \\\hline
9 & $-1 - {\tau}^{2} - {\tau}^{3}$ & 
16 & 
$(-1-\tau, 0, -1)_{\tau}$ & 3 \\\hline
10 & $1 + {\tau}^{2} + {\tau}^{3}$ & 
16 & 
$(1+\tau, 0, 1)_{\tau}$ & 3 \\\hline
11 & $ - {\tau}^{2} - {\tau}^{3}$ & 
20 & 
$(-1-\tau, 0, 0)_{\tau}$ & 3 \\\hline
12 & $ + {\tau}^{2} + {\tau}^{3}$ & 
20 & 
$(1+\tau, 0, 0)_{\tau}$ & 3 \\\hline
13 & $-3 - {\tau} - {\tau}^{3}$ & 
16 & 
$(1, 0, 0, 0, 1+\tau)_{\tau}$ & 5 \\\hline
14 & $3 + {\tau} + {\tau}^{3}$ & 
16 & 
$(-1, 0, 0, 0, -1-\tau)_{\tau}$ & 5 \\\hline
15 & $-2 - {\tau} - {\tau}^{3}$ & 
14 & 
$(1, 0, 0, 0, 2+\tau)_{\tau}$ & 5 \\\hline
16 & $2 + {\tau} + {\tau}^{3}$ & 
14 & 
$(1, 0, 0, 2+\tau)_{\tau}$ & 4 \\\hline
17 & $-1 - {\tau} - {\tau}^{3}$ & 
16 & 
$(-1, 0, 0, -1-\tau)_{\tau}$ & 4 \\\hline
18 & $1 + {\tau} + {\tau}^{3}$ & 
16 & 
$(1, 0, 0, 1+\tau)_{\tau}$ & 4 \\\hline
19 & $-4 - {\tau}^{3}$ & 
20 & 
$(1, 0, 0, 2, 0)_{\tau}$ & 5 \\\hline
20 & $4 + {\tau}^{3}$ & 
20 & 
$(-1, 0, 0, -2, 0)_{\tau}$ & 5 \\\hline
21 & $-3 - {\tau}^{3}$ & 
13 & 
$(-1, 0, 0, -2, 0, 1-2\tau)_{\tau}$ & 6 \\\hline
22 & $3 + {\tau}^{3}$ & 
13 & 
$(-1, 0, 0, 0, -1-2\tau)_{\tau}$ & 5 \\\hline
23 & $-2 - {\tau}^{3}$ & 
10 & 
$(-1, 0, 0, -2)_{\tau}$ & 4 \\\hline
24 & $2 + {\tau}^{3}$ & 
10 & 
$(1, 0, 0, 2)_{\tau}$ & 4 \\\hline
25 & $-1 - {\tau}^{3}$ & 
11 & 
$(-1, 0, 0, -1)_{\tau}$ & 4 \\\hline
26 & $1 + {\tau}^{3}$ & 
11 & 
$(1, 0, 0, 1)_{\tau}$ & 4 \\\hline
27 & $ - {\tau}^{3}$ & 
16 & 
$(-1, 0, 0, 0)_{\tau}$ & 4 \\\hline
28 & $ + {\tau}^{3}$ & 
16 & 
$(1, 0, 0, 0)_{\tau}$ & 4 \\\hline
29 & $-3 + {\tau} - {\tau}^{3}$ & 
18 & 
$(-1, 0, 0, -2, 0, 1-\tau)_{\tau}$ & 6 \\\hline
30 & $3 - {\tau} + {\tau}^{3}$ & 
18 & 
$(1, 0, 0, 2, 0, -1+\tau)_{\tau}$ & 6 \\\hline
31 & $-2 + {\tau} - {\tau}^{3}$ & 
14 & 
$(-1, 0, 0, -2, 0, 2-\tau)_{\tau}$ & 6 \\\hline
32 & $2 - {\tau} + {\tau}^{3}$ & 
14 & 
$(1, 0, 0, 2-\tau)_{\tau}$ & 4 \\\hline
33 & $-1 + {\tau} - {\tau}^{3}$ & 
14 & 
$(-1, 0, 0, -1+\tau)_{\tau}$ & 4 \\\hline
34 & $1 - {\tau} + {\tau}^{3}$ & 
14 & 
$(1, 0, 0, 1-\tau)_{\tau}$ & 4 \\\hline
35 & $ + {\tau} - {\tau}^{3}$ & 
18 & 
$(-1, 0, 1, 0)_{\tau}$ & 4 \\\hline
36 & $ - {\tau} + {\tau}^{3}$ & 
18 & 
$(1, 0, -1, 0)_{\tau}$ & 4 \\\hline
37 & $-2 + {\tau}^{2} - {\tau}^{3}$ & 
20 & 
$(1-\tau, 0, -2)_{\tau}$ & 3 \\\hline
38 & $2 - {\tau}^{2} + {\tau}^{3}$ & 
20 & 
$(-1+\tau, 0, 2)_{\tau}$ & 3 \\\hline
39 & $ - 2{\tau} - {\tau}^{2}$ & 
20 & 
$(1+\tau, 0, 0, 2+\tau, 0)_{\tau}$ & 5 \\\hline
40 & $ + 2{\tau} + {\tau}^{2}$ & 
20 & 
$(2+\tau, 0)_{\tau}$ & 2 \\\hline
41 & $-2 - {\tau} - {\tau}^{2}$ & 
18 & 
$(1, 0, -1+\tau, 0, 2+\tau)_{\tau}$ & 5 \\\hline
42 & $2 + {\tau} + {\tau}^{2}$ & 
18 & 
$(1, 0, 2+\tau)_{\tau}$ & 3 \\\hline
43 & $-1 - {\tau} - {\tau}^{2}$ & 
12 & 
$(-1, 0, -1-\tau)_{\tau}$ & 3 \\\hline
44 & $1 + {\tau} + {\tau}^{2}$ & 
12 & 
$(1, 0, 1+\tau)_{\tau}$ & 3 \\\hline
45 & $ - {\tau} - {\tau}^{2}$ & 
10 & 
$(-1-\tau, 0)_{\tau}$ & 2 \\\hline
46 & $ + {\tau} + {\tau}^{2}$ & 
10 & 
$(1+\tau, 0)_{\tau}$ & 2 \\\hline
47 & $1 - {\tau} - {\tau}^{2}$ & 
12 & 
$(-1, 0, 1-\tau)_{\tau}$ & 3 \\\hline
48 & $-1 + {\tau} + {\tau}^{2}$ & 
12 & 
$(1, 0, -1+\tau)_{\tau}$ & 3 \\\hline
49 & $2 - {\tau} - {\tau}^{2}$ & 
18 & 
$(-1, 0, 2-\tau)_{\tau}$ & 3 \\\hline
50 & $-2 + {\tau} + {\tau}^{2}$ & 
18 & 
$(-1, 0, 0, -1+\tau, 0, 2-\tau)_{\tau}$ & 6 \\\hline
51 & $-2 - {\tau}^{2}$ & 
18 & 
$(-1, 0, -2)_{\tau}$ & 3 \\\hline
52 & $2 + {\tau}^{2}$ & 
18 & 
$(1, 0, 2)_{\tau}$ & 3 \\\hline
53 & $-1 - {\tau}^{2}$ & 
11 & 
$(-1, 0, -1)_{\tau}$ & 3 \\\hline
54 & $1 + {\tau}^{2}$ & 
11 & 
$(1, 0, 1)_{\tau}$ & 3 \\\hline
55 & $ - {\tau}^{2}$ & 
8 & 
$(-1, 0, 0)_{\tau}$ & 3 \\\hline
56 & $ + {\tau}^{2}$ & 
8 & 
$(1, 0, 0)_{\tau}$ & 3 \\\hline
57 & $1 - {\tau}^{2}$ & 
9 & 
$(-1, 0, 1)_{\tau}$ & 3 \\\hline
58 & $-1 + {\tau}^{2}$ & 
9 & 
$(1, 0, -1)_{\tau}$ & 3 \\\hline
59 & $2 - {\tau}^{2}$ & 
14 & 
$(-1, 0, 2)_{\tau}$ & 3 \\\hline
60 & $-2 + {\tau}^{2}$ & 
14 & 
$(1, 0, -2)_{\tau}$ & 3 \\\hline
61 & $-1 + {\tau} - {\tau}^{2}$ & 
18 & 
$(-1, 0, -1+\tau)_{\tau}$ & 3 \\\hline
62 & $1 - {\tau} + {\tau}^{2}$ & 
18 & 
$(1, 0, 1-\tau)_{\tau}$ & 3 \\\hline
63 & $ + {\tau} - {\tau}^{2}$ & 
14 & 
$(1-\tau, 0)_{\tau}$ & 2 \\\hline
64 & $ - {\tau} + {\tau}^{2}$ & 
14 & 
$(-1+\tau, 0)_{\tau}$ & 2 \\\hline
65 & $1 + {\tau} - {\tau}^{2}$ & 
14 & 
$(-1, 0, 1+\tau)_{\tau}$ & 3 \\\hline
66 & $-1 - {\tau} + {\tau}^{2}$ & 
14 & 
$(1, 0, -1-\tau)_{\tau}$ & 3 \\\hline
67 & $2 + {\tau} - {\tau}^{2}$ & 
18 & 
$(-1, 0, 2+\tau)_{\tau}$ & 3 \\\hline
68 & $-2 - {\tau} + {\tau}^{2}$ & 
18 & 
$(1, 0, 1+\tau, 0, 2+\tau)_{\tau}$ & 5 \\\hline
69 & $-2 - 2{\tau}$ & 
20 & 
$(1+\tau, 0, 0, 2)_{\tau}$ & 4 \\\hline
70 & $2 + 2{\tau}$ & 
20 & 
$(-1-\tau, 0, 0, -2)_{\tau}$ & 4 \\\hline
71 & $-1 - 2{\tau}$ & 
16 & 
$(-1-2\tau)_{\tau}$ & 1 \\\hline
72 & $1 + 2{\tau}$ & 
16 & 
$(-1-\tau, 0, -2, 0, 1-2\tau)_{\tau}$ & 5 \\\hline
73 & $ - 2{\tau}$ & 
16 & 
$(-2, 0)_{\tau}$ & 2 \\\hline
74 & $ + 2{\tau}$ & 
16 & 
$(2, 0)_{\tau}$ & 2 \\\hline
75 & $1 - 2{\tau}$ & 
20 & 
$(1-2\tau)_{\tau}$ & 1 \\\hline
76 & $-1 + 2{\tau}$ & 
20 & 
$(-1-\tau, 0, -2, 0, -1-2\tau)_{\tau}$ & 5 \\\hline
77 & $-3 - {\tau}$ & 
19 & 
$(1+\tau, 0, 0, 1+\tau)_{\tau}$ & 4 \\\hline
78 & $3 + {\tau}$ & 
19 & 
$(-1-\tau, 0, 0, -1-\tau)_{\tau}$ & 4 \\\hline
79 & $-2 - {\tau}$ & 
10 & 
$(1+\tau, 0, 0, 2+\tau)_{\tau}$ & 4 \\\hline
80 & $2 + {\tau}$ & 
10 & 
$(2+\tau)_{\tau}$ & 1 \\\hline
81 & $-1 - {\tau}$ & 
5 & 
$(-1-\tau)_{\tau}$ & 1 \\\hline
82 & $1 + {\tau}$ & 
5 & 
$(1+\tau)_{\tau}$ & 1 \\\hline
83 & $ - {\tau}$ & 
4 & 
$(-1, 0)_{\tau}$ & 2 \\\hline
84 & $ + {\tau}$ & 
4 & 
$(1, 0)_{\tau}$ & 2 \\\hline
85 & $1 - {\tau}$ & 
7 & 
$(1-\tau)_{\tau}$ & 1 \\\hline
86 & $-1 + {\tau}$ & 
7 & 
$(-1+\tau)_{\tau}$ & 1 \\\hline
87 & $2 - {\tau}$ & 
14 & 
$(2-\tau)_{\tau}$ & 1 \\\hline
88 & $-2 + {\tau}$ & 
14 & 
$(-1, 0, 0, -2, 0, 2-\tau, 0, 2-\tau)_{\tau}$ & 8 \\\hline
89 & $-3$ & 
18 & 
$(-1, 0, 0, -2, 0, 2-\tau, 0, 1-2\tau)_{\tau}$ & 8 \\\hline
90 & $3$ & 
18 & 
$(-1-\tau, 0, 0, -1-2\tau)_{\tau}$ & 4 \\\hline
91 & $-2$ & 
8 & 
$(-2)_{\tau}$ & 1 \\\hline
92 & $2$ & 
8 & 
$(2)_{\tau}$ & 1 \\\hline
93 & $-1$ & 
2 & 
$(-1)_{\tau}$ & 1 \\\hline
94 & $1$ & 
2 & 
$(1)_{\tau}$ & 1 \\\hline
\end{longtable}
\end{center}

From Table~\ref{table:hyperKC_tnaf_existence_D1_mu=1} 
(or Table~\ref{table:hyperKC_tnaf_existence_D1_mu=-1}), 
there are 94 elements whose square values of norm are less than or equal to $20$. 
The ${\tau}$-NAF of ${\alpha}^{\prime}$ with 
$\mathcal{N}({\alpha}^{\prime}) \leq 20$ are also shown 
in Table~\ref{table:hyperKC_tnaf_existence_D1_mu=1} 
and Table~\ref{table:hyperKC_tnaf_existence_D1_mu=-1} 
(in the case of the digit set $\widetilde{\mathcal{D}_{1}}$). 
For the other cases, namely $j = 2, \ldots, 16$, 
similar results can be obtained 
(See Table~\ref{table:hyperKC_tnaf_existence_D2_mu=1}%
--\ref{table:hyperKC_tnaf_existence_D16_mu=-1} 
in Appendix~\ref{sec:hyperKC_append_tnaf_existence}). 
The proof is complete. 
\end{proof}

\begin{thm} 
\label{thm:hyperKC_tnaf_unique}
{\rm\bfseries 
[Uniqueness of ${\tau}$-NAF]}\ 
For each $j$ $(j = 1, \ldots, 16)$, 
it holds the uniqueness for ${\tau}$-NAF %
with respect to the digit set $\widetilde{\mathcal{D}_{j}}$. 
\end{thm}

\begin{proof}
Let us assume the contrary and seek a contradiction. 
Suppose that there exists an element $\alpha \in \mathbb{Z}[\tau]$ 
which has two different such representations 
\[
\alpha = %
\sum_{i=0}^{\ell-1} {c_i} {\tau}^{i} = %
\sum_{i=0}^{{\ell}^{\prime}-1} {b_i} {\tau}^{i}, \quad 
{c_i} = c_i^{\prime} + c_i^{\prime\prime}{\tau}, \ %
{b_i} = b_i^{\prime} + b_i^{\prime\prime}{\tau} %
\in \widetilde{\mathcal{D}_{j}}. %
\]
By padding zeros if necessary, 
we can rewrite the above equation as 
\[
\sum_{i=0}^{\ell-1} {c_i} {\tau}^{i} = %
\sum_{i=0}^{{\ell}-1} {b_i} {\tau}^{i}. %
\]
Let ${i_0}:= \min \{ i \in \{0, 1, \ldots, \ell-1\} \mid %
{b_i} \neq {c_i} \}$. 
Replacing $\sum_{i=0}^{\ell-1} {c_i} {\tau}^{i}$ 
by $(\sum_{i=0}^{\ell-1} {c_i} {\tau}^{i} %
- \sum_{i=0}^{{i_0}-1} {c_i} {\tau}^{i})/{\tau}^{i_0}$ 
and 
$\sum_{i=0}^{\ell-1} {b_i} {\tau}^{i}$ 
by $(\sum_{i=0}^{\ell-1} {b_i} {\tau}^{i} %
- \sum_{i=0}^{{i_0}-1} {b_i} {\tau}^{i})/{\tau}^{i_0}$, 
we may assume that ${c_0} \neq {b_0}$. 
From Lemma~\ref{lem:hyperKC_div}, we must have 
${c_0} \neq 0$ and ${b_0} \neq 0$. 
We also have ${c_1} = {b_1} = 0$. 
Then $-({c_0} - {b_0}) = \sum_{i=2}^{\ell-1} %
({c_i} - {b_i}) {\tau}^{i}$ is divisible by ${\tau}^{2}$. 
Hence by Lemma~\ref{lem:hyperKC_div}, 
we obtain $4 \mid (b_0^{\prime} - c_0^{\prime})$ and 
$4 \mid (\mu(b_0^{\prime} - c_0^{\prime})/2 + (b_0^{\prime\prime} - c_0^{\prime\prime}))$. 
From Condition~3, we immediately have 
$b_0^{\prime} - c_0^{\prime} = 0$ or $\pm{4}$. 
We shall see that in each case, there is a contradiction. 

\medskip
\noindent
{\bf Case~1. } 
Suppose $b_0^{\prime} - c_0^{\prime} = 0$. 
From $4 \mid (\mu(b_0^{\prime} - c_0^{\prime})/2 + (b_0^{\prime\prime} - c_0^{\prime\prime}))$, 
we have 
$b_0^{\prime\prime} - c_0^{\prime\prime} = 0$ or $\pm{4}$. 
If $b_0^{\prime\prime} - c_0^{\prime\prime} = 0$, 
then ${c_0} = {b_0}$, contrary to our assumption. 
Thus $b_0^{\prime\prime} - c_0^{\prime\prime} = \pm{4}$ 
implies that 
$b_0^{\prime\prime} = 2$ and $c_0^{\prime\prime} = -2$, 
or $b_0^{\prime\prime} = -2$ and $c_0^{\prime\prime} = 2$. 
Hence ${b_0} = b_0^{\prime} + 2{\tau}$ and 
${c_0} = b_0^{\prime} - 2{\tau}$ or 
${b_0} = b_0^{\prime} - 2{\tau}$ and 
${c_0} = b_0^{\prime} + 2{\tau}$. 
However, by Condition~4, there do not exist 
two elements of the form $b_0^{\prime} \pm{2}{\tau}$ 
in $\widetilde{\mathcal{D}_{j}}$ simultaneously 
for each $j$ ($1 \leq j \leq 16$). 
This is a contradiction. 

\medskip
\noindent
{\bf Case~2. } 
Suppose $b_0^{\prime} - c_0^{\prime} = \pm{4}$. 
Since $\mu(b_0^{\prime} - c_0^{\prime})/2 + (b_0^{\prime\prime} - c_0^{\prime\prime}) = %
\pm{2}{\mu} + (b_0^{\prime\prime} - c_0^{\prime\prime})$, 
we obtain 
$b_0^{\prime\prime} - c_0^{\prime\prime} = \pm{2}{\mu}$. 
Thus $b_0^{\prime\prime} = \pm{\mu}$ and 
$c_0^{\prime\prime} = \mp{\mu}$. 
Consequently 
$({b_0}, {c_0}) = %
(2 + {\mu}{\tau}, -2 - {\mu}{\tau}), 
(2 - {\mu}{\tau}, -2 + {\mu}{\tau}), 
(-2 + {\mu}{\tau}, 2 - {\mu}{\tau}), 
(-2 - {\mu}{\tau}, 2 + {\mu}{\tau})$. 
However, as in case~1, 
it does not satisfy 
${b_0} \in \widetilde{\mathcal{D}_{j}}$ and 
${c_0} \in \widetilde{\mathcal{D}_{j}}$ simultaneously 
for each $j$ ($1 \leq j \leq 16$). 
This is a contradiction. 

\medskip
\noindent
Therefore the proof is complete. 
\end{proof}

Recall that the ${\tau}$-NAF on Koblitz curves has 
existence, uniqueness, and optimality. 
In contrast to the ${\tau}$-NAF on Koblitz curves, 
it does not only hold the optimality. 

\begin{thm} 
\label{thm:hyperKC_tnaf_non_optimality}
{\rm\bfseries 
[Non-minimality of ${\tau}$-NAF]}\ 
It does not hold the minimality for ${\tau}$-NAF %
with respect to the digit set $\widetilde{\mathcal{D}_{j}}$ 
$(j = 1, \ldots, 16)$. 
\end{thm}

\begin{proof}
Consider $\alpha = 2{\mu}{\tau} + 2 = (2\mu, 2)_{\tau} \in \mathbb{Z}[\tau]$. 
Obviously, $(2\mu, 2)_{\tau}$ is a ${\tau}$-adic expansion 
with respect to the digit set $\widetilde{\mathcal{D}_{j}}$ for each $j$. 
The Hamming weight of the ${\tau}$-adic expansion 
$(2\mu, 2)_{\tau}$ is $2$. 
If $2 + {\mu}{\tau} \in \widetilde{\mathcal{D}_{j}}$ for some $j$, 
then $\mu = 1$ and $1 \leq j \leq 8$ or 
$\mu = -1$ and $1 \leq j \leq 4$, $9 \leq j \leq 12$. 
Moreover, it is easy to check that 
$(-\mu, 0, 0, 2 + {\mu}{\tau}, 0, -2)_{\tau}$ is a ${\tau}$-NAF of $\alpha$. 
The Hamming weight is $3$. 
If $2 + {\mu}{\tau} \not\in \widetilde{\mathcal{D}_{j}}$ for some $j$, 
we have $\mu = 1$ and $9 \leq j \leq 16$ or 
$\mu = -1$ and $1 \leq j \leq 4$, $9 \leq j \leq 12$. 
Then one can easily verify that 
$-2 - {\mu}{\tau} \in \widetilde{\mathcal{D}_{j}}$ 
and 
$(-\mu, 0, 0, 2, 0, -2 - {\mu}{\tau}, 0, -2)_{\tau}$ 
is a ${\tau}$-NAF of $\alpha$. 
The Hamming weight is $4$. 
Combining the above two cases yields that 
there exists an element $\alpha \in \mathbb{Z}[\tau]$ 
which does not have minimal Hamming weight 
among all $\tau$-adic expansions 
with respect to the digit set $\widetilde{\mathcal{D}_{j}}$. 
This completes the proof. 
\end{proof}


\section{Conclusion}
\label{sec:hyperKC_conclusion}

In this paper, we explored the three properties, 
namely, the existence, uniqueness, 
and the minimality of the Hamming weight 
for the two classes of ${\tau}$-adic expansions 
on hyperelliptic Koblitz curves. 
While the ${\tau}$-NAF on Koblitz curves 
inherits the outstanding properties of NAF, 
the $\tau$-adic expansion with the strategy 
at least one of four consecutive coefficients is zero, 
only has the existence. 
We also gave a detailed investigation 
for the ${\tau}$-NAF on hyperelliptic Koblitz curves. 
We showed that there exist 16 digit sets 
so that one can achieve the ${\tau}$-NAF. 
Furthermore, we proved that 
the ${\tau}$-NAF on hyperelliptic Koblitz curves 
has the existence and uniqueness. 
But, as is the case with higher width version 
of the $\tau$-NAF on Koblitz curves, 
the ${\tau}$-NAF on hyperelliptic Koblitz curves 
does not have the minimality. 
In order to determine 
which elements in $\mathbb{Z}[{\tau}]$ have minimal Hamming weight, 
a more sophisticated analysis will be required. 


\begin{small}

\end{small}

\appendix

\newpage

\section{Non-uniqueness of GLS expansions}
\label{sec:hyperKC_append_gls_non_uniqueness}

\begin{center}
\begin{longtable}[c]{|c|c||c|c||c|c|}
\caption{$\alpha = ({c_3}, {c_2}, {c_1}, {c_0})_{\tau}$ with %
${c_{i+3}}{c_{i+2}}{c_{i+1}}{c_{i}} = 0$ ($\mu = 1$). }
\\
\hline
\label{table:hyperKC_non_gls_example_1}
${\#}$ & ${\alpha} = ({c_3}, {c_2}, {c_1}, {c_0})_{\tau}$ & %
${\#}$ & ${\alpha} = ({c_3}, {c_2}, {c_1}, {c_0})_{\tau}$ & %
${\#}$ & ${\alpha} = ({c_3}, {c_2}, {c_1}, {c_0})_{\tau}$ \\
\hline\hline
\endfirsthead
\hline
${\#}$ & ${\alpha} = ({c_3}, {c_2}, {c_1}, {c_0})_{\tau}$ & %
${\#}$ & ${\alpha} = ({c_3}, {c_2}, {c_1}, {c_0})_{\tau}$ & %
${\#}$ & ${\alpha} = ({c_3}, {c_2}, {c_1}, {c_0})_{\tau}$ \\
\hline\hline
\endhead
\hline
\endfoot
\hline
\endlastfoot
\hline
$1$ & $(-3, 0, -3, -3)_{\tau}$ & 
$2$ & $(-3, 0, -3, -2)_{\tau}$ & 
$3$ & $(-3, 0, -3, -1)_{\tau}$ \\\hline
$4$ & $(-3, 0, -3, 1)_{\tau}$ & 
$5$ & $(-3, 0, -3, 2)_{\tau}$ & 
$6$ & $(-3, 0, -3, 3)_{\tau}$ \\\hline
$7$ & $(-3, 0, -2, -3)_{\tau}$ & 
$8$ & $(-3, 0, -2, -2)_{\tau}$ & 
$9$ & $(-3, 0, -2, -1)_{\tau}$ \\\hline
$10$ & $(-3, 0, -2, 1)_{\tau}$ & 
$11$ & $(-3, 0, -2, 2)_{\tau}$ & 
$12$ & $(-3, 0, -2, 3)_{\tau}$ \\\hline
$13$ & $(-3, 0, -1, -3)_{\tau}$ & 
$14$ & $(-3, 0, -1, -2)_{\tau}$ & 
$15$ & $(-3, 0, -1, -1)_{\tau}$ \\\hline
$16$ & $(-3, 0, -1, 1)_{\tau}$ & 
$17$ & $(-3, 0, -1, 2)_{\tau}$ & 
$18$ & $(-3, 0, -1, 3)_{\tau}$ \\\hline
$19$ & $(-3, 0, 1, -3)_{\tau}$ & 
$20$ & $(-3, 0, 1, -2)_{\tau}$ & 
$21$ & $(-3, 0, 1, -1)_{\tau}$ \\\hline
$22$ & $(-3, 0, 1, 1)_{\tau}$ & 
$23$ & $(-3, 0, 1, 2)_{\tau}$ & 
$24$ & $(-3, 0, 1, 3)_{\tau}$ \\
\end{longtable}
\end{center}

\begin{center}
\begin{longtable}[c]{|c|c||c|c||c|c|}
\caption{$\alpha = ({c_3}, {c_2}, {c_1}, {c_0})_{\tau}$ with %
${c_{i+3}}{c_{i+2}}{c_{i+1}}{c_{i}} = 0$ ($\mu = -1$). }
\\
\hline
\label{table:hyperKC_non_gls_example_2}
${\#}$ & ${\alpha} = ({c_3}, {c_2}, {c_1}, {c_0})_{\tau}$ & %
${\#}$ & ${\alpha} = ({c_3}, {c_2}, {c_1}, {c_0})_{\tau}$ & %
${\#}$ & ${\alpha} = ({c_3}, {c_2}, {c_1}, {c_0})_{\tau}$ \\
\hline\hline
\endfirsthead
\hline
${\#}$ & ${\alpha} = ({c_3}, {c_2}, {c_1}, {c_0})_{\tau}$ & %
${\#}$ & ${\alpha} = ({c_3}, {c_2}, {c_1}, {c_0})_{\tau}$ & %
${\#}$ & ${\alpha} = ({c_3}, {c_2}, {c_1}, {c_0})_{\tau}$ \\
\hline\hline
\endhead
\hline
\endfoot
\hline
\endlastfoot
\hline
$1$ & $(-3, 0, -3, -3)_{\tau}$ & 
$2$ & $(-3, 0, -3, -2)_{\tau}$ & 
$3$ & $(-3, 0, -3, -1)_{\tau}$ \\\hline
$4$ & $(-3, 0, -3, 1)_{\tau}$ & 
$5$ & $(-3, 0, -3, 2)_{\tau}$ & 
$6$ & $(-3, 0, -3, 3)_{\tau}$ \\\hline
$7$ & $(-3, 0, -2, -3)_{\tau}$ & 
$8$ & $(-3, 0, -2, -2)_{\tau}$ & 
$9$ & $(-3, 0, -2, -1)_{\tau}$ \\\hline
$10$ & $(-3, 0, -2, 1)_{\tau}$ & 
$11$ & $(-3, 0, -2, 2)_{\tau}$ & 
$12$ & $(-3, 0, -2, 3)_{\tau}$ \\\hline
$13$ & $(-3, 0, -1, -3)_{\tau}$ & 
$14$ & $(-3, 0, -1, -2)_{\tau}$ & 
$15$ & $(-3, 0, -1, -1)_{\tau}$ \\\hline
$16$ & $(-3, 0, -1, 1)_{\tau}$ & 
$17$ & $(-3, 0, -1, 2)_{\tau}$ & 
$18$ & $(-3, 0, -1, 3)_{\tau}$ \\\hline
$19$ & $(-3, 0, 1, -3)_{\tau}$ & 
$20$ & $(-3, 0, 1, -2)_{\tau}$ & 
$21$ & $(-3, 0, 1, -1)_{\tau}$ \\\hline
$22$ & $(-3, 0, 1, 1)_{\tau}$ & 
$23$ & $(-3, 0, 1, 2)_{\tau}$ & 
$24$ & $(-3, 0, 1, 3)_{\tau}$ \\
\end{longtable}
\end{center}


\newpage
\section{Existence of ${\tau}$-NAF}
\label{sec:hyperKC_append_tnaf_existence}

\begin{center}
\begin{longtable}[c]{|c||c|c|c|c|}
\caption{Elements $\alpha^{\prime} \in \mathbb{Z}[{\tau}]$ with $\mathcal{N}(\alpha^{\prime})^{2} \leq 20$ 
and their ${\tau}$-NAF ($\widetilde{\mathcal{D}_{2}}$, $\mu = 1$). }
\\
\hline
\label{table:hyperKC_tnaf_existence_D2_mu=1}
${\#}$ & ${\alpha^{\prime}} = s + t{\tau} + u{\tau}^{2} + v{\tau}^{3}$ & %
$\mathcal{N}(\alpha^{\prime})^{2}$ & ${\tau}$-NAF of $\alpha^{\prime}$ & $\ell(\alpha^{\prime})$ \\
\hline\hline
\endfirsthead
\hline
${\#}$ & ${\alpha^{\prime}} = s + t{\tau} + u{\tau}^{2} + v{\tau}^{3}$ & %
$\mathcal{N}(\alpha^{\prime})^{2}$ & ${\tau}$-NAF of $\alpha^{\prime}$ & $\ell(\alpha^{\prime})$ \\
\hline\hline
\endhead
\hline
\endfoot
\hline
\endlastfoot
\hline
1 & $2 - {\tau}^{2} - {\tau}^{3}$ & 
20 & 
$(-1-\tau, 0, 2)_{\tau}$ & 3 \\\hline
2 & $-2 + {\tau}^{2} + {\tau}^{3}$ & 
20 & 
$(1+\tau, 0, -2)_{\tau}$ & 3 \\\hline
3 & $1 - {\tau} - {\tau}^{3}$ & 
16 & 
$(-1, 0, 0, 1-\tau)_{\tau}$ & 4 \\\hline
4 & $-1 + {\tau} + {\tau}^{3}$ & 
16 & 
$(1, 0, 0, -1+\tau)_{\tau}$ & 4 \\\hline
5 & $2 - {\tau} - {\tau}^{3}$ & 
14 & 
$(-1, 0, 0, 2-\tau)_{\tau}$ & 4 \\\hline
6 & $-2 + {\tau} + {\tau}^{3}$ & 
14 & 
$(1, 0, 0, 0, 2-\tau)_{\tau}$ & 5 \\\hline
7 & $3 - {\tau} - {\tau}^{3}$ & 
16 & 
$(-1, 0, 0, 0, -1+\tau)_{\tau}$ & 5 \\\hline
8 & $-3 + {\tau} + {\tau}^{3}$ & 
16 & 
$(1, 0, 0, 0, 1-\tau)_{\tau}$ & 5 \\
\end{longtable}
\end{center}

\begin{center}
\begin{longtable}[c]{|c||c|c|c|c|}
\caption{Elements $\alpha^{\prime} \in \mathbb{Z}[{\tau}]$ with $\mathcal{N}(\alpha^{\prime})^{2} \leq 20$ 
and their ${\tau}$-NAF ($\widetilde{\mathcal{D}_{2}}$, $\mu = -1$). }
\\
\hline
\label{table:hyperKC_tnaf_existence_D2_mu=-1}
${\#}$ & ${\alpha^{\prime}} = s + t{\tau} + u{\tau}^{2} + v{\tau}^{3}$ & %
$\mathcal{N}(\alpha^{\prime})^{2}$ & ${\tau}$-NAF of $\alpha^{\prime}$ & $\ell(\alpha^{\prime})$ \\
\hline\hline
\endfirsthead
\hline
${\#}$ & ${\alpha^{\prime}} = s + t{\tau} + u{\tau}^{2} + v{\tau}^{3}$ & %
$\mathcal{N}(\alpha^{\prime})^{2}$ & ${\tau}$-NAF of $\alpha^{\prime}$ & $\ell(\alpha^{\prime})$ \\
\hline\hline
\endhead
\hline
\endfoot
\hline
\endlastfoot
\hline
1 & $-2 - {\tau} - {\tau}^{2} - {\tau}^{3}$ & 
18 & 
$(1, 0, -1, 0, 2+\tau)_{\tau}$ & 5 \\\hline
2 & $2 + {\tau} + {\tau}^{2} + {\tau}^{3}$ & 
18 & 
$(1+\tau, 0, 2+\tau)_{\tau}$ & 3 \\\hline
3 & $-1 - {\tau} - {\tau}^{2} - {\tau}^{3}$ & 
19 & 
$(-1-\tau, 0, -1-\tau)_{\tau}$ & 3 \\\hline
4 & $1 + {\tau} + {\tau}^{2} + {\tau}^{3}$ & 
19 & 
$(1+\tau, 0, 1+\tau)_{\tau}$ & 3 \\\hline
5 & $-3 - {\tau}^{2} - {\tau}^{3}$ & 
20 & 
$(-1, 0, 0, -2, 0, 1-2\tau, 0, 1-2\tau)_{\tau}$ & 8 \\\hline
6 & $3 + {\tau}^{2} + {\tau}^{3}$ & 
20 & 
$(1, 0, 0, 2, 0, -1+2\tau, 0, -1+2\tau)_{\tau}$ & 8 \\\hline
7 & $-2 - {\tau}^{2} - {\tau}^{3}$ & 
16 & 
$(-1-\tau, 0, -2)_{\tau}$ & 3 \\\hline
8 & $2 + {\tau}^{2} + {\tau}^{3}$ & 
16 & 
$(1+\tau, 0, 2)_{\tau}$ & 3 \\
\end{longtable}
\end{center}

\begin{center}
\begin{longtable}[c]{|c||c|c|c|c|}
\caption{Elements $\alpha^{\prime} \in \mathbb{Z}[{\tau}]$ with $\mathcal{N}(\alpha^{\prime})^{2} \leq 20$ 
and their ${\tau}$-NAF ($\widetilde{\mathcal{D}_{3}}$, $\mu = 1$). }
\\
\hline
\label{table:hyperKC_tnaf_existence_D3_mu=1}
${\#}$ & ${\alpha^{\prime}} = s + t{\tau} + u{\tau}^{2} + v{\tau}^{3}$ & %
$\mathcal{N}(\alpha^{\prime})^{2}$ & ${\tau}$-NAF of $\alpha^{\prime}$ & $\ell(\alpha^{\prime})$ \\
\hline\hline
\endfirsthead
\hline
${\#}$ & ${\alpha^{\prime}} = s + t{\tau} + u{\tau}^{2} + v{\tau}^{3}$ & %
$\mathcal{N}(\alpha^{\prime})^{2}$ & ${\tau}$-NAF of $\alpha^{\prime}$ & $\ell(\alpha^{\prime})$ \\
\hline\hline
\endhead
\hline
\endfoot
\hline
\endlastfoot
\hline
1 & $2 - {\tau}^{2} - {\tau}^{3}$ & 
20 & 
$(-1-\tau, 0, 2)_{\tau}$ & 3 \\\hline
2 & $-2 + {\tau}^{2} + {\tau}^{3}$ & 
20 & 
$(1+\tau, 0, -2)_{\tau}$ & 3 \\\hline
3 & $1 - {\tau} - {\tau}^{3}$ & 
16 & 
$(-1, 0, 0, 1-\tau)_{\tau}$ & 4 \\\hline
4 & $-1 + {\tau} + {\tau}^{3}$ & 
16 & 
$(1, 0, 0, -1+\tau)_{\tau}$ & 4 \\\hline
5 & $2 - {\tau} - {\tau}^{3}$ & 
14 & 
$(-1, 0, 0, 2-\tau)_{\tau}$ & 4 \\\hline
6 & $-2 + {\tau} + {\tau}^{3}$ & 
14 & 
$(1, 0, 0, 0, 2-\tau)_{\tau}$ & 5 \\\hline
7 & $3 - {\tau} - {\tau}^{3}$ & 
16 & 
$(-1, 0, 0, 0, -1+\tau)_{\tau}$ & 5 \\\hline
8 & $-3 + {\tau} + {\tau}^{3}$ & 
16 & 
$(1, 0, 0, 0, 1-\tau)_{\tau}$ & 5 \\
\end{longtable}
\end{center}

\begin{center}
\begin{longtable}[c]{|c||c|c|c|c|}
\caption{Elements $\alpha^{\prime} \in \mathbb{Z}[{\tau}]$ with $\mathcal{N}(\alpha^{\prime})^{2} \leq 20$ 
and their ${\tau}$-NAF ($\widetilde{\mathcal{D}_{3}}$, $\mu = -1$). }
\\
\hline
\label{table:hyperKC_tnaf_existence_D3_mu=-1}
${\#}$ & ${\alpha^{\prime}} = s + t{\tau} + u{\tau}^{2} + v{\tau}^{3}$ & %
$\mathcal{N}(\alpha^{\prime})^{2}$ & ${\tau}$-NAF of $\alpha^{\prime}$ & $\ell(\alpha^{\prime})$ \\
\hline\hline
\endfirsthead
\hline
${\#}$ & ${\alpha^{\prime}} = s + t{\tau} + u{\tau}^{2} + v{\tau}^{3}$ & %
$\mathcal{N}(\alpha^{\prime})^{2}$ & ${\tau}$-NAF of $\alpha^{\prime}$ & $\ell(\alpha^{\prime})$ \\
\hline\hline
\endhead
\hline
\endfoot
\hline
\endlastfoot
\hline
1 & $-2 - {\tau} - {\tau}^{2} - {\tau}^{3}$ & 
18 & 
$(1, 0, -1, 0, 2+\tau)_{\tau}$ & 5 \\\hline
2 & $2 + {\tau} + {\tau}^{2} + {\tau}^{3}$ & 
18 & 
$(1+\tau, 0, 2+\tau)_{\tau}$ & 3 \\\hline
3 & $-1 - {\tau} - {\tau}^{2} - {\tau}^{3}$ & 
19 & 
$(-1-\tau, 0, -1-\tau)_{\tau}$ & 3 \\\hline
4 & $1 + {\tau} + {\tau}^{2} + {\tau}^{3}$ & 
19 & 
$(1+\tau, 0, 1+\tau)_{\tau}$ & 3 \\\hline
5 & $-3 - {\tau}^{2} - {\tau}^{3}$ & 
20 & 
$(1, 0, -1, 0, 1+2\tau)_{\tau}$ & 5 \\\hline
6 & $3 + {\tau}^{2} + {\tau}^{3}$ & 
20 & 
$(-1, 0, 1, 0, -1-2\tau)_{\tau}$ & 5 \\\hline
7 & $-2 - {\tau}^{2} - {\tau}^{3}$ & 
16 & 
$(-1-\tau, 0, -2)_{\tau}$ & 3 \\\hline
8 & $2 + {\tau}^{2} + {\tau}^{3}$ & 
16 & 
$(1+\tau, 0, 2)_{\tau}$ & 3 \\
\end{longtable}
\end{center}

\begin{center}
\begin{longtable}[c]{|c||c|c|c|c|}
\caption{Elements $\alpha^{\prime} \in \mathbb{Z}[{\tau}]$ with $\mathcal{N}(\alpha^{\prime})^{2} \leq 20$ 
and their ${\tau}$-NAF ($\widetilde{\mathcal{D}_{4}}$, $\mu = 1$). }
\\
\hline
\label{table:hyperKC_tnaf_existence_D4_mu=1}
${\#}$ & ${\alpha^{\prime}} = s + t{\tau} + u{\tau}^{2} + v{\tau}^{3}$ & %
$\mathcal{N}(\alpha^{\prime})^{2}$ & ${\tau}$-NAF of $\alpha^{\prime}$ & $\ell(\alpha^{\prime})$ \\
\hline\hline
\endfirsthead
\hline
${\#}$ & ${\alpha^{\prime}} = s + t{\tau} + u{\tau}^{2} + v{\tau}^{3}$ & %
$\mathcal{N}(\alpha^{\prime})^{2}$ & ${\tau}$-NAF of $\alpha^{\prime}$ & $\ell(\alpha^{\prime})$ \\
\hline\hline
\endhead
\hline
\endfoot
\hline
\endlastfoot
\hline
1 & $2 - {\tau}^{2} - {\tau}^{3}$ & 
20 & 
$(-1-\tau, 0, 2)_{\tau}$ & 3 \\\hline
2 & $-2 + {\tau}^{2} + {\tau}^{3}$ & 
20 & 
$(1+\tau, 0, -2)_{\tau}$ & 3 \\\hline
3 & $1 - {\tau} - {\tau}^{3}$ & 
16 & 
$(-1, 0, 0, 1-\tau)_{\tau}$ & 4 \\\hline
4 & $-1 + {\tau} + {\tau}^{3}$ & 
16 & 
$(1, 0, 0, -1+\tau)_{\tau}$ & 4 \\\hline
5 & $2 - {\tau} - {\tau}^{3}$ & 
14 & 
$(-1, 0, 0, 2-\tau)_{\tau}$ & 4 \\\hline
6 & $-2 + {\tau} + {\tau}^{3}$ & 
14 & 
$(1, 0, 0, 0, 2-\tau)_{\tau}$ & 5 \\\hline
7 & $3 - {\tau} - {\tau}^{3}$ & 
16 & 
$(-1, 0, 0, 0, -1+\tau)_{\tau}$ & 5 \\\hline
8 & $-3 + {\tau} + {\tau}^{3}$ & 
16 & 
$(1, 0, 0, 0, 1-\tau)_{\tau}$ & 5 \\
\end{longtable}
\end{center}

\begin{center}
\begin{longtable}[c]{|c||c|c|c|c|}
\caption{Elements $\alpha^{\prime} \in \mathbb{Z}[{\tau}]$ with $\mathcal{N}(\alpha^{\prime})^{2} \leq 20$ 
and their ${\tau}$-NAF ($\widetilde{\mathcal{D}_{4}}$, $\mu = -1$). }
\\
\hline
\label{table:hyperKC_tnaf_existence_D4_mu=-1}
${\#}$ & ${\alpha^{\prime}} = s + t{\tau} + u{\tau}^{2} + v{\tau}^{3}$ & %
$\mathcal{N}(\alpha^{\prime})^{2}$ & ${\tau}$-NAF of $\alpha^{\prime}$ & $\ell(\alpha^{\prime})$ \\
\hline\hline
\endfirsthead
\hline
${\#}$ & ${\alpha^{\prime}} = s + t{\tau} + u{\tau}^{2} + v{\tau}^{3}$ & %
$\mathcal{N}(\alpha^{\prime})^{2}$ & ${\tau}$-NAF of $\alpha^{\prime}$ & $\ell(\alpha^{\prime})$ \\
\hline\hline
\endhead
\hline
\endfoot
\hline
\endlastfoot
\hline
1 & $-2 - {\tau} - {\tau}^{2} - {\tau}^{3}$ & 
18 & 
$(1, 0, -1, 0, 2+\tau)_{\tau}$ & 5 \\\hline
2 & $2 + {\tau} + {\tau}^{2} + {\tau}^{3}$ & 
18 & 
$(1+\tau, 0, 2+\tau)_{\tau}$ & 3 \\\hline
3 & $-1 - {\tau} - {\tau}^{2} - {\tau}^{3}$ & 
19 & 
$(-1-\tau, 0, -1-\tau)_{\tau}$ & 3 \\\hline
4 & $1 + {\tau} + {\tau}^{2} + {\tau}^{3}$ & 
19 & 
$(1+\tau, 0, 1+\tau)_{\tau}$ & 3 \\\hline
5 & $-3 - {\tau}^{2} - {\tau}^{3}$ & 
20 & 
$(1, 0, -1, 0, 1+2\tau)_{\tau}$ & 5 \\\hline
6 & $3 + {\tau}^{2} + {\tau}^{3}$ & 
20 & 
$(1, 0, 0, 2, 0, -1+2\tau, 0, -1+2\tau)_{\tau}$ & 8 \\\hline
7 & $-2 - {\tau}^{2} - {\tau}^{3}$ & 
16 & 
$(-1-\tau, 0, -2)_{\tau}$ & 3 \\\hline
8 & $2 + {\tau}^{2} + {\tau}^{3}$ & 
16 & 
$(1+\tau, 0, 2)_{\tau}$ & 3 \\
\end{longtable}
\end{center}

\begin{center}
\begin{longtable}[c]{|c||c|c|c|c|}
\caption{Elements $\alpha^{\prime} \in \mathbb{Z}[{\tau}]$ with $\mathcal{N}(\alpha^{\prime})^{2} \leq 20$ 
and their ${\tau}$-NAF ($\widetilde{\mathcal{D}_{5}}$, $\mu = 1$). }
\\
\hline
\label{table:hyperKC_tnaf_existence_D5_mu=1}
${\#}$ & ${\alpha^{\prime}} = s + t{\tau} + u{\tau}^{2} + v{\tau}^{3}$ & %
$\mathcal{N}(\alpha^{\prime})^{2}$ & ${\tau}$-NAF of $\alpha^{\prime}$ & $\ell(\alpha^{\prime})$ \\
\hline\hline
\endfirsthead
\hline
${\#}$ & ${\alpha^{\prime}} = s + t{\tau} + u{\tau}^{2} + v{\tau}^{3}$ & %
$\mathcal{N}(\alpha^{\prime})^{2}$ & ${\tau}$-NAF of $\alpha^{\prime}$ & $\ell(\alpha^{\prime})$ \\
\hline\hline
\endhead
\hline
\endfoot
\hline
\endlastfoot
\hline
1 & $2 - {\tau}^{2} - {\tau}^{3}$ & 
20 & 
$(-1-\tau, 0, 2)_{\tau}$ & 3 \\\hline
2 & $-2 + {\tau}^{2} + {\tau}^{3}$ & 
20 & 
$(1+\tau, 0, -2)_{\tau}$ & 3 \\\hline
3 & $1 - {\tau} - {\tau}^{3}$ & 
16 & 
$(-1, 0, 0, 1-\tau)_{\tau}$ & 4 \\\hline
4 & $-1 + {\tau} + {\tau}^{3}$ & 
16 & 
$(1, 0, 0, -1+\tau)_{\tau}$ & 4 \\\hline
5 & $2 - {\tau} - {\tau}^{3}$ & 
14 & 
$(-1, 0, 0, 0, -2+\tau)_{\tau}$ & 5 \\\hline
6 & $-2 + {\tau} + {\tau}^{3}$ & 
14 & 
$(1, 0, 0, -2+\tau)_{\tau}$ & 4 \\\hline
7 & $3 - {\tau} - {\tau}^{3}$ & 
16 & 
$(-1, 0, 0, 0, -1+\tau)_{\tau}$ & 5 \\\hline
8 & $-3 + {\tau} + {\tau}^{3}$ & 
16 & 
$(1, 0, 0, 0, 1-\tau)_{\tau}$ & 5 \\
\end{longtable}
\end{center}

\begin{center}
\begin{longtable}[c]{|c||c|c|c|c|}
\caption{Elements $\alpha^{\prime} \in \mathbb{Z}[{\tau}]$ with $\mathcal{N}(\alpha^{\prime})^{2} \leq 20$ 
and their ${\tau}$-NAF ($\widetilde{\mathcal{D}_{5}}$, $\mu = -1$). }
\\
\hline
\label{table:hyperKC_tnaf_existence_D5_mu=-1}
${\#}$ & ${\alpha^{\prime}} = s + t{\tau} + u{\tau}^{2} + v{\tau}^{3}$ & %
$\mathcal{N}(\alpha^{\prime})^{2}$ & ${\tau}$-NAF of $\alpha^{\prime}$ & $\ell(\alpha^{\prime})$ \\
\hline\hline
\endfirsthead
\hline
${\#}$ & ${\alpha^{\prime}} = s + t{\tau} + u{\tau}^{2} + v{\tau}^{3}$ & %
$\mathcal{N}(\alpha^{\prime})^{2}$ & ${\tau}$-NAF of $\alpha^{\prime}$ & $\ell(\alpha^{\prime})$ \\
\hline\hline
\endhead
\hline
\endfoot
\hline
\endlastfoot
\hline
1 & $-2 - {\tau} - {\tau}^{2} - {\tau}^{3}$ & 
18 & 
$(1, 0, -1, 0, 2+\tau)_{\tau}$ & 5 \\\hline
2 & $2 + {\tau} + {\tau}^{2} + {\tau}^{3}$ & 
18 & 
$(1+\tau, 0, 2+\tau)_{\tau}$ & 3 \\\hline
3 & $-1 - {\tau} - {\tau}^{2} - {\tau}^{3}$ & 
19 & 
$(-1-\tau, 0, -1-\tau)_{\tau}$ & 3 \\\hline
4 & $1 + {\tau} + {\tau}^{2} + {\tau}^{3}$ & 
19 & 
$(1+\tau, 0, 1+\tau)_{\tau}$ & 3 \\\hline
5 & $-3 - {\tau}^{2} - {\tau}^{3}$ & 
20 & 
$(-1, 0, 0, -2, 0, 1-2\tau, 0, 1-2\tau)_{\tau}$ & 8 \\\hline
6 & $3 + {\tau}^{2} + {\tau}^{3}$ & 
20 & 
$(-1, 0, 1, 0, -1-2\tau)_{\tau}$ & 5 \\\hline
7 & $-2 - {\tau}^{2} - {\tau}^{3}$ & 
16 & 
$(-1-\tau, 0, -2)_{\tau}$ & 3 \\\hline
8 & $2 + {\tau}^{2} + {\tau}^{3}$ & 
16 & 
$(1+\tau, 0, 2)_{\tau}$ & 3 \\
\end{longtable}
\end{center}

\begin{center}
\begin{longtable}[c]{|c||c|c|c|c|}
\caption{Elements $\alpha^{\prime} \in \mathbb{Z}[{\tau}]$ with $\mathcal{N}(\alpha^{\prime})^{2} \leq 20$ 
and their ${\tau}$-NAF ($\widetilde{\mathcal{D}_{6}}$, $\mu = 1$). }
\\
\hline
\label{table:hyperKC_tnaf_existence_D6_mu=1}
${\#}$ & ${\alpha^{\prime}} = s + t{\tau} + u{\tau}^{2} + v{\tau}^{3}$ & %
$\mathcal{N}(\alpha^{\prime})^{2}$ & ${\tau}$-NAF of $\alpha^{\prime}$ & $\ell(\alpha^{\prime})$ \\
\hline\hline
\endfirsthead
\hline
${\#}$ & ${\alpha^{\prime}} = s + t{\tau} + u{\tau}^{2} + v{\tau}^{3}$ & %
$\mathcal{N}(\alpha^{\prime})^{2}$ & ${\tau}$-NAF of $\alpha^{\prime}$ & $\ell(\alpha^{\prime})$ \\
\hline\hline
\endhead
\hline
\endfoot
\hline
\endlastfoot
\hline
1 & $2 - {\tau}^{2} - {\tau}^{3}$ & 
20 & 
$(-1-\tau, 0, 2)_{\tau}$ & 3 \\\hline
2 & $-2 + {\tau}^{2} + {\tau}^{3}$ & 
20 & 
$(1+\tau, 0, -2)_{\tau}$ & 3 \\\hline
3 & $1 - {\tau} - {\tau}^{3}$ & 
16 & 
$(-1, 0, 0, 1-\tau)_{\tau}$ & 4 \\\hline
4 & $-1 + {\tau} + {\tau}^{3}$ & 
16 & 
$(1, 0, 0, -1+\tau)_{\tau}$ & 4 \\\hline
5 & $2 - {\tau} - {\tau}^{3}$ & 
14 & 
$(-1, 0, 0, 0, -2+\tau)_{\tau}$ & 5 \\\hline
6 & $-2 + {\tau} + {\tau}^{3}$ & 
14 & 
$(1, 0, 0, -2+\tau)_{\tau}$ & 4 \\\hline
7 & $3 - {\tau} - {\tau}^{3}$ & 
16 & 
$(-1, 0, 0, 0, -1+\tau)_{\tau}$ & 5 \\\hline
8 & $-3 + {\tau} + {\tau}^{3}$ & 
16 & 
$(1, 0, 0, 0, 1-\tau)_{\tau}$ & 5 \\
\end{longtable}
\end{center}

\begin{center}
\begin{longtable}[c]{|c||c|c|c|c|}
\caption{Elements $\alpha^{\prime} \in \mathbb{Z}[{\tau}]$ with $\mathcal{N}(\alpha^{\prime})^{2} \leq 20$ 
and their ${\tau}$-NAF ($\widetilde{\mathcal{D}_{6}}$, $\mu = -1$). }
\\
\hline
\label{table:hyperKC_tnaf_existence_D6_mu=-1}
${\#}$ & ${\alpha^{\prime}} = s + t{\tau} + u{\tau}^{2} + v{\tau}^{3}$ & %
$\mathcal{N}(\alpha^{\prime})^{2}$ & ${\tau}$-NAF of $\alpha^{\prime}$ & $\ell(\alpha^{\prime})$ \\
\hline\hline
\endfirsthead
\hline
${\#}$ & ${\alpha^{\prime}} = s + t{\tau} + u{\tau}^{2} + v{\tau}^{3}$ & %
$\mathcal{N}(\alpha^{\prime})^{2}$ & ${\tau}$-NAF of $\alpha^{\prime}$ & $\ell(\alpha^{\prime})$ \\
\hline\hline
\endhead
\hline
\endfoot
\hline
\endlastfoot
\hline
1 & $-2 - {\tau} - {\tau}^{2} - {\tau}^{3}$ & 
18 & 
$(1, 0, -1, 0, 2+\tau)_{\tau}$ & 5 \\\hline
2 & $2 + {\tau} + {\tau}^{2} + {\tau}^{3}$ & 
18 & 
$(1+\tau, 0, 2+\tau)_{\tau}$ & 3 \\\hline
3 & $-1 - {\tau} - {\tau}^{2} - {\tau}^{3}$ & 
19 & 
$(-1-\tau, 0, -1-\tau)_{\tau}$ & 3 \\\hline
4 & $1 + {\tau} + {\tau}^{2} + {\tau}^{3}$ & 
19 & 
$(1+\tau, 0, 1+\tau)_{\tau}$ & 3 \\\hline
5 & $-3 - {\tau}^{2} - {\tau}^{3}$ & 
20 & 
$(-1, 0, 0, -2, 0, 1-2\tau, 0, 1-2\tau)_{\tau}$ & 8 \\\hline
6 & $3 + {\tau}^{2} + {\tau}^{3}$ & 
20 & 
$(1, 0, 0, 2, 0, -1+2\tau, 0, -1+2\tau)_{\tau}$ & 8 \\\hline
7 & $-2 - {\tau}^{2} - {\tau}^{3}$ & 
16 & 
$(-1-\tau, 0, -2)_{\tau}$ & 3 \\\hline
8 & $2 + {\tau}^{2} + {\tau}^{3}$ & 
16 & 
$(1+\tau, 0, 2)_{\tau}$ & 3 \\
\end{longtable}
\end{center}

\begin{center}
\begin{longtable}[c]{|c||c|c|c|c|}
\caption{Elements $\alpha^{\prime} \in \mathbb{Z}[{\tau}]$ with $\mathcal{N}(\alpha^{\prime})^{2} \leq 20$ 
and their ${\tau}$-NAF ($\widetilde{\mathcal{D}_{7}}$, $\mu = 1$). }
\\
\hline
\label{table:hyperKC_tnaf_existence_D7_mu=1}
${\#}$ & ${\alpha^{\prime}} = s + t{\tau} + u{\tau}^{2} + v{\tau}^{3}$ & %
$\mathcal{N}(\alpha^{\prime})^{2}$ & ${\tau}$-NAF of $\alpha^{\prime}$ & $\ell(\alpha^{\prime})$ \\
\hline\hline
\endfirsthead
\hline
${\#}$ & ${\alpha^{\prime}} = s + t{\tau} + u{\tau}^{2} + v{\tau}^{3}$ & %
$\mathcal{N}(\alpha^{\prime})^{2}$ & ${\tau}$-NAF of $\alpha^{\prime}$ & $\ell(\alpha^{\prime})$ \\
\hline\hline
\endhead
\hline
\endfoot
\hline
\endlastfoot
\hline
1 & $2 - {\tau}^{2} - {\tau}^{3}$ & 
20 & 
$(-1-\tau, 0, 2)_{\tau}$ & 3 \\\hline
2 & $-2 + {\tau}^{2} + {\tau}^{3}$ & 
20 & 
$(1+\tau, 0, -2)_{\tau}$ & 3 \\\hline
3 & $1 - {\tau} - {\tau}^{3}$ & 
16 & 
$(-1, 0, 0, 1-\tau)_{\tau}$ & 4 \\\hline
4 & $-1 + {\tau} + {\tau}^{3}$ & 
16 & 
$(1, 0, 0, -1+\tau)_{\tau}$ & 4 \\\hline
5 & $2 - {\tau} - {\tau}^{3}$ & 
14 & 
$(-1, 0, 0, 0, -2+\tau)_{\tau}$ & 5 \\\hline
6 & $-2 + {\tau} + {\tau}^{3}$ & 
14 & 
$(1, 0, 0, -2+\tau)_{\tau}$ & 4 \\\hline
7 & $3 - {\tau} - {\tau}^{3}$ & 
16 & 
$(-1, 0, 0, 0, -1+\tau)_{\tau}$ & 5 \\\hline
8 & $-3 + {\tau} + {\tau}^{3}$ & 
16 & 
$(1, 0, 0, 0, 1-\tau)_{\tau}$ & 5 \\
\end{longtable}
\end{center}

\begin{center}
\begin{longtable}[c]{|c||c|c|c|c|}
\caption{Elements $\alpha^{\prime} \in \mathbb{Z}[{\tau}]$ with $\mathcal{N}(\alpha^{\prime})^{2} \leq 20$ 
and their ${\tau}$-NAF ($\widetilde{\mathcal{D}_{7}}$, $\mu = -1$). }
\\
\hline
\label{table:hyperKC_tnaf_existence_D7_mu=-1}
${\#}$ & ${\alpha^{\prime}} = s + t{\tau} + u{\tau}^{2} + v{\tau}^{3}$ & %
$\mathcal{N}(\alpha^{\prime})^{2}$ & ${\tau}$-NAF of $\alpha^{\prime}$ & $\ell(\alpha^{\prime})$ \\
\hline\hline
\endfirsthead
\hline
${\#}$ & ${\alpha^{\prime}} = s + t{\tau} + u{\tau}^{2} + v{\tau}^{3}$ & %
$\mathcal{N}(\alpha^{\prime})^{2}$ & ${\tau}$-NAF of $\alpha^{\prime}$ & $\ell(\alpha^{\prime})$ \\
\hline\hline
\endhead
\hline
\endfoot
\hline
\endlastfoot
\hline
1 & $-2 - {\tau} - {\tau}^{2} - {\tau}^{3}$ & 
18 & 
$(1, 0, -1, 0, 2+\tau)_{\tau}$ & 5 \\\hline
2 & $2 + {\tau} + {\tau}^{2} + {\tau}^{3}$ & 
18 & 
$(1+\tau, 0, 2+\tau)_{\tau}$ & 3 \\\hline
3 & $-1 - {\tau} - {\tau}^{2} - {\tau}^{3}$ & 
19 & 
$(-1-\tau, 0, -1-\tau)_{\tau}$ & 3 \\\hline
4 & $1 + {\tau} + {\tau}^{2} + {\tau}^{3}$ & 
19 & 
$(1+\tau, 0, 1+\tau)_{\tau}$ & 3 \\\hline
5 & $-3 - {\tau}^{2} - {\tau}^{3}$ & 
20 & 
$(1, 0, -1, 0, 1+2\tau)_{\tau}$ & 5 \\\hline
6 & $3 + {\tau}^{2} + {\tau}^{3}$ & 
20 & 
$(-1, 0, 1, 0, -1-2\tau)_{\tau}$ & 5 \\\hline
7 & $-2 - {\tau}^{2} - {\tau}^{3}$ & 
16 & 
$(-1-\tau, 0, -2)_{\tau}$ & 3 \\\hline
8 & $2 + {\tau}^{2} + {\tau}^{3}$ & 
16 & 
$(1+\tau, 0, 2)_{\tau}$ & 3 \\
\end{longtable}
\end{center}

\begin{center}
\begin{longtable}[c]{|c||c|c|c|c|}
\caption{Elements $\alpha^{\prime} \in \mathbb{Z}[{\tau}]$ with $\mathcal{N}(\alpha^{\prime})^{2} \leq 20$ 
and their ${\tau}$-NAF ($\widetilde{\mathcal{D}_{8}}$, $\mu = 1$). }
\\
\hline
\label{table:hyperKC_tnaf_existence_D8_mu=1}
${\#}$ & ${\alpha^{\prime}} = s + t{\tau} + u{\tau}^{2} + v{\tau}^{3}$ & %
$\mathcal{N}(\alpha^{\prime})^{2}$ & ${\tau}$-NAF of $\alpha^{\prime}$ & $\ell(\alpha^{\prime})$ \\
\hline\hline
\endfirsthead
\hline
${\#}$ & ${\alpha^{\prime}} = s + t{\tau} + u{\tau}^{2} + v{\tau}^{3}$ & %
$\mathcal{N}(\alpha^{\prime})^{2}$ & ${\tau}$-NAF of $\alpha^{\prime}$ & $\ell(\alpha^{\prime})$ \\
\hline\hline
\endhead
\hline
\endfoot
\hline
\endlastfoot
\hline
1 & $2 - {\tau}^{2} - {\tau}^{3}$ & 
20 & 
$(-1-\tau, 0, 2)_{\tau}$ & 3 \\\hline
2 & $-2 + {\tau}^{2} + {\tau}^{3}$ & 
20 & 
$(1+\tau, 0, -2)_{\tau}$ & 3 \\\hline
3 & $1 - {\tau} - {\tau}^{3}$ & 
16 & 
$(-1, 0, 0, 1-\tau)_{\tau}$ & 4 \\\hline
4 & $-1 + {\tau} + {\tau}^{3}$ & 
16 & 
$(1, 0, 0, -1+\tau)_{\tau}$ & 4 \\\hline
5 & $2 - {\tau} - {\tau}^{3}$ & 
14 & 
$(-1, 0, 0, 0, -2+\tau)_{\tau}$ & 5 \\\hline
6 & $-2 + {\tau} + {\tau}^{3}$ & 
14 & 
$(1, 0, 0, -2+\tau)_{\tau}$ & 4 \\\hline
7 & $3 - {\tau} - {\tau}^{3}$ & 
16 & 
$(-1, 0, 0, 0, -1+\tau)_{\tau}$ & 5 \\\hline
8 & $-3 + {\tau} + {\tau}^{3}$ & 
16 & 
$(1, 0, 0, 0, 1-\tau)_{\tau}$ & 5 \\
\end{longtable}
\end{center}

\begin{center}
\begin{longtable}[c]{|c||c|c|c|c|}
\caption{Elements $\alpha^{\prime} \in \mathbb{Z}[{\tau}]$ with $\mathcal{N}(\alpha^{\prime})^{2} \leq 20$ 
and their ${\tau}$-NAF ($\widetilde{\mathcal{D}_{8}}$, $\mu = -1$). }
\\
\hline
\label{table:hyperKC_tnaf_existence_D8_mu=-1}
${\#}$ & ${\alpha^{\prime}} = s + t{\tau} + u{\tau}^{2} + v{\tau}^{3}$ & %
$\mathcal{N}(\alpha^{\prime})^{2}$ & ${\tau}$-NAF of $\alpha^{\prime}$ & $\ell(\alpha^{\prime})$ \\
\hline\hline
\endfirsthead
\hline
${\#}$ & ${\alpha^{\prime}} = s + t{\tau} + u{\tau}^{2} + v{\tau}^{3}$ & %
$\mathcal{N}(\alpha^{\prime})^{2}$ & ${\tau}$-NAF of $\alpha^{\prime}$ & $\ell(\alpha^{\prime})$ \\
\hline\hline
\endhead
\hline
\endfoot
\hline
\endlastfoot
\hline
1 & $-2 - {\tau} - {\tau}^{2} - {\tau}^{3}$ & 
18 & 
$(1, 0, -1, 0, 2+\tau)_{\tau}$ & 5 \\\hline
2 & $2 + {\tau} + {\tau}^{2} + {\tau}^{3}$ & 
18 & 
$(1+\tau, 0, 2+\tau)_{\tau}$ & 3 \\\hline
3 & $-1 - {\tau} - {\tau}^{2} - {\tau}^{3}$ & 
19 & 
$(-1-\tau, 0, -1-\tau)_{\tau}$ & 3 \\\hline
4 & $1 + {\tau} + {\tau}^{2} + {\tau}^{3}$ & 
19 & 
$(1+\tau, 0, 1+\tau)_{\tau}$ & 3 \\\hline
5 & $-3 - {\tau}^{2} - {\tau}^{3}$ & 
20 & 
$(1, 0, -1, 0, 1+2\tau)_{\tau}$ & 5 \\\hline
6 & $3 + {\tau}^{2} + {\tau}^{3}$ & 
20 & 
$(1, 0, 0, 2, 0, -1+2\tau, 0, -1+2\tau)_{\tau}$ & 8 \\\hline
7 & $-2 - {\tau}^{2} - {\tau}^{3}$ & 
16 & 
$(-1-\tau, 0, -2)_{\tau}$ & 3 \\\hline
8 & $2 + {\tau}^{2} + {\tau}^{3}$ & 
16 & 
$(1+\tau, 0, 2)_{\tau}$ & 3 \\
\end{longtable}
\end{center}

\begin{center}
\begin{longtable}[c]{|c||c|c|c|c|}
\caption{Elements $\alpha^{\prime} \in \mathbb{Z}[{\tau}]$ with $\mathcal{N}(\alpha^{\prime})^{2} \leq 20$ 
and their ${\tau}$-NAF ($\widetilde{\mathcal{D}_{9}}$, $\mu = 1$). }
\\
\hline
\label{table:hyperKC_tnaf_existence_D9_mu=1}
${\#}$ & ${\alpha^{\prime}} = s + t{\tau} + u{\tau}^{2} + v{\tau}^{3}$ & %
$\mathcal{N}(\alpha^{\prime})^{2}$ & ${\tau}$-NAF of $\alpha^{\prime}$ & $\ell(\alpha^{\prime})$ \\
\hline\hline
\endfirsthead
\hline
${\#}$ & ${\alpha^{\prime}} = s + t{\tau} + u{\tau}^{2} + v{\tau}^{3}$ & %
$\mathcal{N}(\alpha^{\prime})^{2}$ & ${\tau}$-NAF of $\alpha^{\prime}$ & $\ell(\alpha^{\prime})$ \\
\hline\hline
\endhead
\hline
\endfoot
\hline
\endlastfoot
\hline
1 & $2 - {\tau}^{2} - {\tau}^{3}$ & 
20 & 
$(-1-\tau, 0, 2)_{\tau}$ & 3 \\\hline
2 & $-2 + {\tau}^{2} + {\tau}^{3}$ & 
20 & 
$(1+\tau, 0, -2)_{\tau}$ & 3 \\\hline
3 & $1 - {\tau} - {\tau}^{3}$ & 
16 & 
$(-1, 0, 0, 1-\tau)_{\tau}$ & 4 \\\hline
4 & $-1 + {\tau} + {\tau}^{3}$ & 
16 & 
$(1, 0, 0, -1+\tau)_{\tau}$ & 4 \\\hline
5 & $2 - {\tau} - {\tau}^{3}$ & 
14 & 
$(-1, 0, 0, 2-\tau)_{\tau}$ & 4 \\\hline
6 & $-2 + {\tau} + {\tau}^{3}$ & 
14 & 
$(1, 0, 0, 0, 2-\tau)_{\tau}$ & 5 \\\hline
7 & $3 - {\tau} - {\tau}^{3}$ & 
16 & 
$(-1, 0, 0, 0, -1+\tau)_{\tau}$ & 5 \\\hline
8 & $-3 + {\tau} + {\tau}^{3}$ & 
16 & 
$(1, 0, 0, 0, 1-\tau)_{\tau}$ & 5 \\
\end{longtable}
\end{center}

\begin{center}
\begin{longtable}[c]{|c||c|c|c|c|}
\caption{Elements $\alpha^{\prime} \in \mathbb{Z}[{\tau}]$ with $\mathcal{N}(\alpha^{\prime})^{2} \leq 20$ 
and their ${\tau}$-NAF ($\widetilde{\mathcal{D}_{9}}$, $\mu = -1$). }
\\
\hline
\label{table:hyperKC_tnaf_existence_D9_mu=-1}
${\#}$ & ${\alpha^{\prime}} = s + t{\tau} + u{\tau}^{2} + v{\tau}^{3}$ & %
$\mathcal{N}(\alpha^{\prime})^{2}$ & ${\tau}$-NAF of $\alpha^{\prime}$ & $\ell(\alpha^{\prime})$ \\
\hline\hline
\endfirsthead
\hline
${\#}$ & ${\alpha^{\prime}} = s + t{\tau} + u{\tau}^{2} + v{\tau}^{3}$ & %
$\mathcal{N}(\alpha^{\prime})^{2}$ & ${\tau}$-NAF of $\alpha^{\prime}$ & $\ell(\alpha^{\prime})$ \\
\hline\hline
\endhead
\hline
\endfoot
\hline
\endlastfoot
\hline
1 & $-2 - {\tau} - {\tau}^{2} - {\tau}^{3}$ & 
18 & 
$(-1-\tau, 0, -2-\tau)_{\tau}$ & 3 \\\hline
2 & $2 + {\tau} + {\tau}^{2} + {\tau}^{3}$ & 
18 & 
$(-1, 0, 1, 0, -2-\tau)_{\tau}$ & 5 \\\hline
3 & $-1 - {\tau} - {\tau}^{2} - {\tau}^{3}$ & 
19 & 
$(-1-\tau, 0, -1-\tau)_{\tau}$ & 3 \\\hline
4 & $1 + {\tau} + {\tau}^{2} + {\tau}^{3}$ & 
19 & 
$(1+\tau, 0, 1+\tau)_{\tau}$ & 3 \\\hline
5 & $-3 - {\tau}^{2} - {\tau}^{3}$ & 
20 & 
$(-1, 0, 0, -2, 0, 1-2\tau, 0, 1-2\tau)_{\tau}$ & 8 \\\hline
6 & $3 + {\tau}^{2} + {\tau}^{3}$ & 
20 & 
$(-1, 0, 1, 0, -1-2\tau)_{\tau}$ & 5 \\\hline
7 & $-2 - {\tau}^{2} - {\tau}^{3}$ & 
16 & 
$(-1-\tau, 0, -2)_{\tau}$ & 3 \\\hline
8 & $2 + {\tau}^{2} + {\tau}^{3}$ & 
16 & 
$(1+\tau, 0, 2)_{\tau}$ & 3 \\
\end{longtable}
\end{center}

\begin{center}
\begin{longtable}[c]{|c||c|c|c|c|}
\caption{Elements $\alpha^{\prime} \in \mathbb{Z}[{\tau}]$ with $\mathcal{N}(\alpha^{\prime})^{2} \leq 20$ 
and their ${\tau}$-NAF ($\widetilde{\mathcal{D}_{10}}$, $\mu = 1$). }
\\
\hline
\label{table:hyperKC_tnaf_existence_D10_mu=1}
${\#}$ & ${\alpha^{\prime}} = s + t{\tau} + u{\tau}^{2} + v{\tau}^{3}$ & %
$\mathcal{N}(\alpha^{\prime})^{2}$ & ${\tau}$-NAF of $\alpha^{\prime}$ & $\ell(\alpha^{\prime})$ \\
\hline\hline
\endfirsthead
\hline
${\#}$ & ${\alpha^{\prime}} = s + t{\tau} + u{\tau}^{2} + v{\tau}^{3}$ & %
$\mathcal{N}(\alpha^{\prime})^{2}$ & ${\tau}$-NAF of $\alpha^{\prime}$ & $\ell(\alpha^{\prime})$ \\
\hline\hline
\endhead
\hline
\endfoot
\hline
\endlastfoot
\hline
1 & $2 - {\tau}^{2} - {\tau}^{3}$ & 
20 & 
$(-1-\tau, 0, 2)_{\tau}$ & 3 \\\hline
2 & $-2 + {\tau}^{2} + {\tau}^{3}$ & 
20 & 
$(1+\tau, 0, -2)_{\tau}$ & 3 \\\hline
3 & $1 - {\tau} - {\tau}^{3}$ & 
16 & 
$(-1, 0, 0, 1-\tau)_{\tau}$ & 4 \\\hline
4 & $-1 + {\tau} + {\tau}^{3}$ & 
16 & 
$(1, 0, 0, -1+\tau)_{\tau}$ & 4 \\\hline
5 & $2 - {\tau} - {\tau}^{3}$ & 
14 & 
$(-1, 0, 0, 2-\tau)_{\tau}$ & 4 \\\hline
6 & $-2 + {\tau} + {\tau}^{3}$ & 
14 & 
$(1, 0, 0, 0, 2-\tau)_{\tau}$ & 5 \\\hline
7 & $3 - {\tau} - {\tau}^{3}$ & 
16 & 
$(-1, 0, 0, 0, -1+\tau)_{\tau}$ & 5 \\\hline
8 & $-3 + {\tau} + {\tau}^{3}$ & 
16 & 
$(1, 0, 0, 0, 1-\tau)_{\tau}$ & 5 \\
\end{longtable}
\end{center}

\begin{center}
\begin{longtable}[c]{|c||c|c|c|c|}
\caption{Elements $\alpha^{\prime} \in \mathbb{Z}[{\tau}]$ with $\mathcal{N}(\alpha^{\prime})^{2} \leq 20$ 
and their ${\tau}$-NAF ($\widetilde{\mathcal{D}_{10}}$, $\mu = -1$). }
\\
\hline
\label{table:hyperKC_tnaf_existence_D10_mu=-1}
${\#}$ & ${\alpha^{\prime}} = s + t{\tau} + u{\tau}^{2} + v{\tau}^{3}$ & %
$\mathcal{N}(\alpha^{\prime})^{2}$ & ${\tau}$-NAF of $\alpha^{\prime}$ & $\ell(\alpha^{\prime})$ \\
\hline\hline
\endfirsthead
\hline
${\#}$ & ${\alpha^{\prime}} = s + t{\tau} + u{\tau}^{2} + v{\tau}^{3}$ & %
$\mathcal{N}(\alpha^{\prime})^{2}$ & ${\tau}$-NAF of $\alpha^{\prime}$ & $\ell(\alpha^{\prime})$ \\
\hline\hline
\endhead
\hline
\endfoot
\hline
\endlastfoot
\hline
1 & $-2 - {\tau} - {\tau}^{2} - {\tau}^{3}$ & 
18 & 
$(-1-\tau, 0, -2-\tau)_{\tau}$ & 3 \\\hline
2 & $2 + {\tau} + {\tau}^{2} + {\tau}^{3}$ & 
18 & 
$(-1, 0, 1, 0, -2-\tau)_{\tau}$ & 5 \\\hline
3 & $-1 - {\tau} - {\tau}^{2} - {\tau}^{3}$ & 
19 & 
$(-1-\tau, 0, -1-\tau)_{\tau}$ & 3 \\\hline
4 & $1 + {\tau} + {\tau}^{2} + {\tau}^{3}$ & 
19 & 
$(1+\tau, 0, 1+\tau)_{\tau}$ & 3 \\\hline
5 & $-3 - {\tau}^{2} - {\tau}^{3}$ & 
20 & 
$(-1, 0, 0, -2, 0, 1-2\tau, 0, 1-2\tau)_{\tau}$ & 8 \\\hline
6 & $3 + {\tau}^{2} + {\tau}^{3}$ & 
20 & 
$(1, 0, 0, 2, 0, -1+2\tau, 0, -1+2\tau)_{\tau}$ & 8 \\\hline
7 & $-2 - {\tau}^{2} - {\tau}^{3}$ & 
16 & 
$(-1-\tau, 0, -2)_{\tau}$ & 3 \\\hline
8 & $2 + {\tau}^{2} + {\tau}^{3}$ & 
16 & 
$(1+\tau, 0, 2)_{\tau}$ & 3 \\
\end{longtable}
\end{center}

\begin{center}
\begin{longtable}[c]{|c||c|c|c|c|}
\caption{Elements $\alpha^{\prime} \in \mathbb{Z}[{\tau}]$ with $\mathcal{N}(\alpha^{\prime})^{2} \leq 20$ 
and their ${\tau}$-NAF ($\widetilde{\mathcal{D}_{11}}$, $\mu = 1$). }
\\
\hline
\label{table:hyperKC_tnaf_existence_D11_mu=1}
${\#}$ & ${\alpha^{\prime}} = s + t{\tau} + u{\tau}^{2} + v{\tau}^{3}$ & %
$\mathcal{N}(\alpha^{\prime})^{2}$ & ${\tau}$-NAF of $\alpha^{\prime}$ & $\ell(\alpha^{\prime})$ \\
\hline\hline
\endfirsthead
\hline
${\#}$ & ${\alpha^{\prime}} = s + t{\tau} + u{\tau}^{2} + v{\tau}^{3}$ & %
$\mathcal{N}(\alpha^{\prime})^{2}$ & ${\tau}$-NAF of $\alpha^{\prime}$ & $\ell(\alpha^{\prime})$ \\
\hline\hline
\endhead
\hline
\endfoot
\hline
\endlastfoot
\hline
1 & $2 - {\tau}^{2} - {\tau}^{3}$ & 
20 & 
$(-1-\tau, 0, 2)_{\tau}$ & 3 \\\hline
2 & $-2 + {\tau}^{2} + {\tau}^{3}$ & 
20 & 
$(1+\tau, 0, -2)_{\tau}$ & 3 \\\hline
3 & $1 - {\tau} - {\tau}^{3}$ & 
16 & 
$(-1, 0, 0, 1-\tau)_{\tau}$ & 4 \\\hline
4 & $-1 + {\tau} + {\tau}^{3}$ & 
16 & 
$(1, 0, 0, -1+\tau)_{\tau}$ & 4 \\\hline
5 & $2 - {\tau} - {\tau}^{3}$ & 
14 & 
$(-1, 0, 0, 2-\tau)_{\tau}$ & 4 \\\hline
6 & $-2 + {\tau} + {\tau}^{3}$ & 
14 & 
$(1, 0, 0, 0, 2-\tau)_{\tau}$ & 5 \\\hline
7 & $3 - {\tau} - {\tau}^{3}$ & 
16 & 
$(-1, 0, 0, 0, -1+\tau)_{\tau}$ & 5 \\\hline
8 & $-3 + {\tau} + {\tau}^{3}$ & 
16 & 
$(1, 0, 0, 0, 1-\tau)_{\tau}$ & 5 \\
\end{longtable}
\end{center}

\begin{center}
\begin{longtable}[c]{|c||c|c|c|c|}
\caption{Elements $\alpha^{\prime} \in \mathbb{Z}[{\tau}]$ with $\mathcal{N}(\alpha^{\prime})^{2} \leq 20$ 
and their ${\tau}$-NAF ($\widetilde{\mathcal{D}_{11}}$, $\mu = -1$). }
\\
\hline
\label{table:hyperKC_tnaf_existence_D11_mu=-1}
${\#}$ & ${\alpha^{\prime}} = s + t{\tau} + u{\tau}^{2} + v{\tau}^{3}$ & %
$\mathcal{N}(\alpha^{\prime})^{2}$ & ${\tau}$-NAF of $\alpha^{\prime}$ & $\ell(\alpha^{\prime})$ \\
\hline\hline
\endfirsthead
\hline
${\#}$ & ${\alpha^{\prime}} = s + t{\tau} + u{\tau}^{2} + v{\tau}^{3}$ & %
$\mathcal{N}(\alpha^{\prime})^{2}$ & ${\tau}$-NAF of $\alpha^{\prime}$ & $\ell(\alpha^{\prime})$ \\
\hline\hline
\endhead
\hline
\endfoot
\hline
\endlastfoot
\hline
1 & $-2 - {\tau} - {\tau}^{2} - {\tau}^{3}$ & 
18 & 
$(-1-\tau, 0, -2-\tau)_{\tau}$ & 3 \\\hline
2 & $2 + {\tau} + {\tau}^{2} + {\tau}^{3}$ & 
18 & 
$(-1, 0, 1, 0, -2-\tau)_{\tau}$ & 5 \\\hline
3 & $-1 - {\tau} - {\tau}^{2} - {\tau}^{3}$ & 
19 & 
$(-1-\tau, 0, -1-\tau)_{\tau}$ & 3 \\\hline
4 & $1 + {\tau} + {\tau}^{2} + {\tau}^{3}$ & 
19 & 
$(1+\tau, 0, 1+\tau)_{\tau}$ & 3 \\\hline
5 & $-3 - {\tau}^{2} - {\tau}^{3}$ & 
20 & 
$(1, 0, -1, 0, 1+2\tau)_{\tau}$ & 5 \\\hline
6 & $3 + {\tau}^{2} + {\tau}^{3}$ & 
20 & 
$(-1, 0, 1, 0, -1-2\tau)_{\tau}$ & 5 \\\hline
7 & $-2 - {\tau}^{2} - {\tau}^{3}$ & 
16 & 
$(-1-\tau, 0, -2)_{\tau}$ & 3 \\\hline
8 & $2 + {\tau}^{2} + {\tau}^{3}$ & 
16 & 
$(1+\tau, 0, 2)_{\tau}$ & 3 \\
\end{longtable}
\end{center}

\begin{center}
\begin{longtable}[c]{|c||c|c|c|c|}
\caption{Elements $\alpha^{\prime} \in \mathbb{Z}[{\tau}]$ with $\mathcal{N}(\alpha^{\prime})^{2} \leq 20$ 
and their ${\tau}$-NAF ($\widetilde{\mathcal{D}_{12}}$, $\mu = 1$). }
\\
\hline
\label{table:hyperKC_tnaf_existence_D12_mu=1}
${\#}$ & ${\alpha^{\prime}} = s + t{\tau} + u{\tau}^{2} + v{\tau}^{3}$ & %
$\mathcal{N}(\alpha^{\prime})^{2}$ & ${\tau}$-NAF of $\alpha^{\prime}$ & $\ell(\alpha^{\prime})$ \\
\hline\hline
\endfirsthead
\hline
${\#}$ & ${\alpha^{\prime}} = s + t{\tau} + u{\tau}^{2} + v{\tau}^{3}$ & %
$\mathcal{N}(\alpha^{\prime})^{2}$ & ${\tau}$-NAF of $\alpha^{\prime}$ & $\ell(\alpha^{\prime})$ \\
\hline\hline
\endhead
\hline
\endfoot
\hline
\endlastfoot
\hline
1 & $2 - {\tau}^{2} - {\tau}^{3}$ & 
20 & 
$(-1-\tau, 0, 2)_{\tau}$ & 3 \\\hline
2 & $-2 + {\tau}^{2} + {\tau}^{3}$ & 
20 & 
$(1+\tau, 0, -2)_{\tau}$ & 3 \\\hline
3 & $1 - {\tau} - {\tau}^{3}$ & 
16 & 
$(-1, 0, 0, 1-\tau)_{\tau}$ & 4 \\\hline
4 & $-1 + {\tau} + {\tau}^{3}$ & 
16 & 
$(1, 0, 0, -1+\tau)_{\tau}$ & 4 \\\hline
5 & $2 - {\tau} - {\tau}^{3}$ & 
14 & 
$(-1, 0, 0, 2-\tau)_{\tau}$ & 4 \\\hline
6 & $-2 + {\tau} + {\tau}^{3}$ & 
14 & 
$(1, 0, 0, 0, 2-\tau)_{\tau}$ & 5 \\\hline
7 & $3 - {\tau} - {\tau}^{3}$ & 
16 & 
$(-1, 0, 0, 0, -1+\tau)_{\tau}$ & 5 \\\hline
8 & $-3 + {\tau} + {\tau}^{3}$ & 
16 & 
$(1, 0, 0, 0, 1-\tau)_{\tau}$ & 5 \\
\end{longtable}
\end{center}

\begin{center}
\begin{longtable}[c]{|c||c|c|c|c|}
\caption{Elements $\alpha^{\prime} \in \mathbb{Z}[{\tau}]$ with $\mathcal{N}(\alpha^{\prime})^{2} \leq 20$ 
and their ${\tau}$-NAF ($\widetilde{\mathcal{D}_{12}}$, $\mu = -1$). }
\\
\hline
\label{table:hyperKC_tnaf_existence_D12_mu=-1}
${\#}$ & ${\alpha^{\prime}} = s + t{\tau} + u{\tau}^{2} + v{\tau}^{3}$ & %
$\mathcal{N}(\alpha^{\prime})^{2}$ & ${\tau}$-NAF of $\alpha^{\prime}$ & $\ell(\alpha^{\prime})$ \\
\hline\hline
\endfirsthead
\hline
${\#}$ & ${\alpha^{\prime}} = s + t{\tau} + u{\tau}^{2} + v{\tau}^{3}$ & %
$\mathcal{N}(\alpha^{\prime})^{2}$ & ${\tau}$-NAF of $\alpha^{\prime}$ & $\ell(\alpha^{\prime})$ \\
\hline\hline
\endhead
\hline
\endfoot
\hline
\endlastfoot
\hline
1 & $-2 - {\tau} - {\tau}^{2} - {\tau}^{3}$ & 
18 & 
$(-1-\tau, 0, -2-\tau)_{\tau}$ & 3 \\\hline
2 & $2 + {\tau} + {\tau}^{2} + {\tau}^{3}$ & 
18 & 
$(-1, 0, 1, 0, -2-\tau)_{\tau}$ & 5 \\\hline
3 & $-1 - {\tau} - {\tau}^{2} - {\tau}^{3}$ & 
19 & 
$(-1-\tau, 0, -1-\tau)_{\tau}$ & 3 \\\hline
4 & $1 + {\tau} + {\tau}^{2} + {\tau}^{3}$ & 
19 & 
$(1+\tau, 0, 1+\tau)_{\tau}$ & 3 \\\hline
5 & $-3 - {\tau}^{2} - {\tau}^{3}$ & 
20 & 
$(1, 0, -1, 0, 1+2\tau)_{\tau}$ & 5 \\\hline
6 & $3 + {\tau}^{2} + {\tau}^{3}$ & 
20 & 
$(1, 0, 0, 2, 0, -1+2\tau, 0, -1+2\tau)_{\tau}$ & 8 \\\hline
7 & $-2 - {\tau}^{2} - {\tau}^{3}$ & 
16 & 
$(-1-\tau, 0, -2)_{\tau}$ & 3 \\\hline
8 & $2 + {\tau}^{2} + {\tau}^{3}$ & 
16 & 
$(1+\tau, 0, 2)_{\tau}$ & 3 \\
\end{longtable}
\end{center}

\begin{center}
\begin{longtable}[c]{|c||c|c|c|c|}
\caption{Elements $\alpha^{\prime} \in \mathbb{Z}[{\tau}]$ with $\mathcal{N}(\alpha^{\prime})^{2} \leq 20$ 
and their ${\tau}$-NAF ($\widetilde{\mathcal{D}_{13}}$, $\mu = 1$). }
\\
\hline
\label{table:hyperKC_tnaf_existence_D13_mu=1}
${\#}$ & ${\alpha^{\prime}} = s + t{\tau} + u{\tau}^{2} + v{\tau}^{3}$ & %
$\mathcal{N}(\alpha^{\prime})^{2}$ & ${\tau}$-NAF of $\alpha^{\prime}$ & $\ell(\alpha^{\prime})$ \\
\hline\hline
\endfirsthead
\hline
${\#}$ & ${\alpha^{\prime}} = s + t{\tau} + u{\tau}^{2} + v{\tau}^{3}$ & %
$\mathcal{N}(\alpha^{\prime})^{2}$ & ${\tau}$-NAF of $\alpha^{\prime}$ & $\ell(\alpha^{\prime})$ \\
\hline\hline
\endhead
\hline
\endfoot
\hline
\endlastfoot
\hline
1 & $2 - {\tau}^{2} - {\tau}^{3}$ & 
20 & 
$(-1-\tau, 0, 2)_{\tau}$ & 3 \\\hline
2 & $-2 + {\tau}^{2} + {\tau}^{3}$ & 
20 & 
$(1+\tau, 0, -2)_{\tau}$ & 3 \\\hline
3 & $1 - {\tau} - {\tau}^{3}$ & 
16 & 
$(-1, 0, 0, 1-\tau)_{\tau}$ & 4 \\\hline
4 & $-1 + {\tau} + {\tau}^{3}$ & 
16 & 
$(1, 0, 0, -1+\tau)_{\tau}$ & 4 \\\hline
5 & $2 - {\tau} - {\tau}^{3}$ & 
14 & 
$(-1, 0, 0, 0, -2+\tau)_{\tau}$ & 5 \\\hline
6 & $-2 + {\tau} + {\tau}^{3}$ & 
14 & 
$(1, 0, 0, -2+\tau)_{\tau}$ & 4 \\\hline
7 & $3 - {\tau} - {\tau}^{3}$ & 
16 & 
$(-1, 0, 0, 0, -1+\tau)_{\tau}$ & 5 \\\hline
8 & $-3 + {\tau} + {\tau}^{3}$ & 
16 & 
$(1, 0, 0, 0, 1-\tau)_{\tau}$ & 5 \\
\end{longtable}%
\end{center}



\end{document}